\documentclass[
]{mpi2015-cscpreprint}

\usepackage[utf8]{inputenc}
\usepackage{tikz}
\usepackage{amsmath}
\usepackage{graphicx}
\usepackage{amsthm}
\usepackage{amssymb}
\usepackage{amsfonts}
\usepackage{bbm}

\usepackage{algorithm}
\usepackage{algpseudocode}
\usepackage{array,multirow}
\usepackage{siunitx}
\usepackage{csquotes}

\usepackage{biblatex}
\usepackage{hyperref}
\usepackage{xurl}
\hypersetup{breaklinks=true}
\usepackage{color}
\usepackage{url}
\usepackage{doi}

\newtheorem{definition}{Definition}[section]
\newtheorem{theorem}{Theorem}[section]
\newtheorem{proposition}{Proposition}[section]
\newtheorem{lemma}[theorem]{Lemma}

\newcommand{\junk}[1] {}

\usepackage{todonotes}

\newcommand{\norm}[1]{\left\lVert#1\right\rVert}
\def\IR{{\mathbb R}}
\def\IC{{\mathbb C}}

\def\IL{{\mathbb L}}

\newcommand{\bA}{{\textbf A}}
\newcommand{\bB}{{\textbf B}}
\newcommand{\bC}{{\textbf C}}
\newcommand{\bD}{{\textbf D}}
\newcommand{\bE}{{\textbf E}}

\newcommand{\bS}{{\textbf S}}
\newcommand{\bY}{{\textbf Y}}
\newcommand{\bJ}{{\textbf J}}
\newcommand{\bK}{{\textbf K}}

\newcommand{\bL}{{\textbf L}}

\newcommand{\bI}{{\textbf I}}

\newcommand{\bP}{{\textbf P}}

\newcommand{\bQ}{{\textbf Q}}
\newcommand{\bW}{{\textbf W}}
\newcommand{\bR}{{\textbf R}}
\newcommand{\bT}{{\textbf T}}

\newcommand{\bX}{{\textbf X}}
\newcommand{\bx}{{\textbf x}}

\newcommand{\bV}{{\textbf V}}
\newcommand{\bU}{{\textbf U}}

\newcommand{\bfe}{{\textbf e}}

\newcommand{\bq}{{\textbf q}}

\newcommand{\bw}{{\textbf w}}

\newcommand{\cD}{ {\cal D} }

\newcommand{\bSigma}{\boldsymbol{\Sigma}}
\newcommand{\bDelta}{\boldsymbol{\Delta}}
\newcommand{\bLambda}{\boldsymbol{\Lambda}}

\newcommand{\Aden}{{\mathcal{A}}}
\newcommand{\Bden}{{\mathcal{B}}}
\newcommand{\Cden}{{\mathcal{C}}}

\def\IR{{\mathbb R}}
\def\IC{{\mathbb C}}
\def\IL{{\mathbb L}}
\def\IV{{\mathbb V}}
\def\IW{{\mathbb W}}
\newcommand{\sIL}{{{{\mathbb L}_s}}}

\newcommand{\Sig}{\boldsymbol{\Sigma}}

\newcommand{\trans}{\ensuremath{^{\mkern-1.5mu\mathsf{T}}}}

\begin{document}

\title{A modified AAA algorithm for learning stable reduced-order models from data}

\author[$\ast$]{Tommaso Bradde}
\affil[$\ast$]{Department of Electronics and Telecommunications,
	Politecnico di Torino, 10129, Italy.\authorcr
	\email{tommaso.bradde@polito.it}, \orcid{0000-0002-9464-2535}}

\author[$\dag$]{Stefano Grivet-Talocia}
\affil[$\dag$]{Department of Electronics and Telecommunications,
	Politecnico di Torino, 10129, Italy.\authorcr
	\email{stefano.grivet@polito.it}, \orcid{0000-0002-5463-3810}}

\author[$\ddagger$]{Quirin Aumann}
\affil[$\ddagger$]{Max Planck Institute for Dynamics of Complex Technical Systems,
	Sandtorstr. 1, 39106 Magdeburg, Germany.\authorcr
	\email{aumann@mpi-magdeburg.mpg.de}, \orcid{0000-0001-7942-5703}}

\author[$\S$]{Ion Victor Gosea}
\affil[$\S$]{Max Planck Institute for Dynamics of Complex Technical Systems,
	Sandtorstr. 1, 39106 Magdeburg, Germany.\authorcr
	\email{gosea@mpi-magdeburg.mpg.de}, \orcid{0000-0003-3580-4116}}

\abstract{In recent years, the Adaptive Antoulas-Anderson (\texttt{AAA}) algorithm has established itself as the method of choice for solving rational approximation problems. Data-driven Model Order Reduction (MOR) of large-scale Linear Time-Invariant (LTI) systems represents one of the many applications in which this algorithm has proven to be successful since it typically generates reduced-order models (ROMs) efficiently and in an automated way. Despite its effectiveness and numerical reliability, the classical \texttt{AAA} algorithm is not guaranteed to return a ROM that retains the same structural features of the underlying dynamical system, such as the stability of the dynamics. In this paper, we propose a novel algebraic characterization for the stability of ROMs with transfer function obeying the \texttt{AAA} barycentric structure. We use this characterization to formulate a set of convex constraints on the free coefficients of the \texttt{AAA} model that, whenever verified, guarantee by construction the asymptotic stability of the resulting ROM. We suggest how to embed such constraints within the \texttt{AAA} optimization routine, and we validate experimentally the effectiveness of the resulting algorithm, named \texttt{stabAAA}, over a set of relevant MOR applications.
}

\novelty{This paper introduces a modified \texttt{AAA} algorithm that can be used to generate asymptotically stable reduced-order models for large-scale linear time-invariant systems. The approach is based on convex programming and generates the model in a completely automated manner.}

\maketitle
\section{Introduction}
\label{sec:intro}

Approximation of large-scale dynamical systems \cite{ACA05} is dire for achieving various goals such as employing efficient simulation or devising robust, automatic control laws. MOR is viewed as a collection of techniques for reducing the computational complexity of mathematical models in numerical simulations. MOR is closely related to the concept of meta-modeling, having various applications in many areas of mathematical modelling.\footnote{\url{https://en.wikipedia.org/wiki/Model_order_reduction}}. Each MOR approach is generally tailored to specific applications and for achieving different goals. Out of the many various ways to reduce large-scale models, two particular classes stand out. They have to do with the explicit accessibility to the high-fidelity model and hence can be categorized into intrusive and non-intrusive methods. In the intrusive case (in which an explicit model, i.e., of ordinary differential equations is available), methods such as balanced truncation (BT) (see \cite{ACA05} and the recent survey \cite{breiten20212}) and moment matching (MM) (\cite{ACA05} and the recent survey \cite{benner2021model}) are available. BT approximates the original system by computing balancing coordinates, and by truncating the components with less energy (offering an a priori error bound, and with a guarantee of stability). MM chooses significant dynamics of the transfer function and interpolates between these features.

On the other hand, the ever-increasing availability of data, i.e., measurements related to the original model, initiate non-intrusive techniques such as Machine Learning (ML) combined with model-based methods \cite{brunton2022data}. ML has demonstrated fairly extensive success in specific tasks, e.g., pattern recognition or feature extraction. The limitations of ML methods might arise when the interpretation of the derived models is under consideration. Therefore, model-based data assimilation through MOR techniques such as dynamic mode decomposition (DMD)~\cite{schmid_2010,ProBruKut2016}, operator inference (OpInf) \cite{morPehW16,morBenGKetal20}, or sparse identification of dynamical systems (SINDY) \cite{SINDY} have become popular. The reason is that, in some cases, an accurate description of the original large-scale dynamical system may not be fully available. Instead, one may have access to different types of measurements  (snapshots in the time domain or in the frequency domain, such as frequency response data). The aforementioned methods use state-access snapshot measurements to achieve model discovery that could be input-dependent. Model identification can be also achieved without using full-state access measurements and by constructing input-invariant models. In this category, system identification methods \cite{Lj99} such as subspace identification approaches \cite{van1994n4sid,van2012subspace} represent non-intrusive methods that deal directly with input-output data (in the time domain) and can identify linear and nonlinear systems, by also offering the opportunity for complexity reduction. Another way to identify a ROM from input-output data is by employing rational approximation when using frequency-domain data (samples of the transfer function, e.g., sampled on the imaginary axis). We refer the reader to the recent book \cite{ABG20} for more details. Other methods also emerged in the last few years; these are mostly based on applying optimization tools or on an adaptive selection of interpolation points, such as the contributions \cite{pradovera2022technique,cherifi2022greedy,goyal2023rank}.

Besides classical rational approximation methods such as the Loewner Framework (\texttt{LF})~\cite{MA07}, and the Vector Fitting (\texttt{VF}) algorithm \cite{VF,BraddeRTVF}, in recent years there has been rapid development of new methods, such as \texttt{RKFIT} \cite{berljafa2017rkfit} and the \texttt{AAA} algorithm \cite{NST18}, to mention only some. The latter was intensively and successfully applied to many challenging problems and test cases, having the advantage of performing the rational approximation task in an automated and numerically reliable way. By performing a greedy iteration guided by a maximum error criterion, this algorithm repeatedly increases the order of the barycentric model structure proposed in~\cite{AA86}, until a desired approximation error on the data values is enforced (based on a pre-defined tolerance value).

Most of the rational approximation algorithms (such as \texttt{VF} and \texttt{RKFIT}) require the user to provide the number of poles to be used in the approximation. Then, the precise values of the poles are found through various optimization techniques, and the best rational approximation, on the data set, is returned. For such approaches in their basic implementation, no particular guarantees for final accuracy in the approximation are enforced. The algorithm needs to be run again with a different starting point (the original guess of the poles) if insufficient accuracy is achieved. More advanced formulations exist for the \texttt{VF} scheme~\cite{grivet2015passive}, that adopt a greedy procedure to automatically determine the model order based on a target accuracy, including dedicated procedures to handle noisy data~\cite{jnl-2006-temc-VFAS,instrumental}. These methods are now part of virtually every Electronic Design Automation (EDA) commercial tool, which are almost invariably based on some \texttt{VF} variant due to its inherent robustness~\cite{mwcl-2008-Deschrijver-FastVF,jnl-2011-tcpmt-pvf,6470726}.

Automatic order estimation through the \texttt{AAA} algorithm is more elegant and straightforward so that this algorithm has the potential to become the method of choice for accuracy-controlled rational interpolation/approximation. Unfortunately,
the \texttt{AAA} algorithm also has some shortcomings. For example, if the data (transfer function measurements) comes from sampling the response of a dynamical system at certain driving frequencies, then important properties such as stability or passivity (corresponding to the original, large-scale model) may not be conserved for the fitted \texttt{AAA} (low-order) model. This behavior has also been reported for the LF, and it is, in general, a problem for most interpolation-, moment matching-, and least squares-based MOR methods \cite{ABG20}. On the other hand, stability enforcement in \texttt{VF} is guaranteed since part of the basic algorithm, which is undoubtedly a key factor for the \texttt{VF} success in engineering applications (see~\cite{grivet2015passive} and references therein).

Stability-preserving and passivity-preserving methods have been developed in the last two decades to cope with the need to construct reliable reduced models that are endowed with stability or passivity properties, being able to be accurately used as surrogates for the full-order models while retaining their properties. Many of these methods are based on constrained optimization or perturbation approaches, often applied a posteriori once stability or passivity violations are identified. A non-exhaustive list of approaches includes the following \cite{grivet2004passivity,antoulas2005new,triverio2007stability,morIonRA08,morReiS11,morReiW11,morPanJWetal13,morGosA16,morBenGV20,morBreU22,goyal2023inference,goyal2023guaranteed}, see also the reviews in~\cite{grivet2015passive} and~\cite[Chapter~5]{mon-2020-MOR-Vol1}.

The goal of this work is to propose a variation of the \texttt{AAA} algorithm that enforces the asymptotic stability of the computed model by construction, making the \texttt{AAA} approach even more appealing for engineering practice. 
Our derivations are inspired by previous works concerning stable MOR in the Sanathanan-Koerner (SK) framework~\cite{BraddeBernstein,ZancoPositive}. In these works, the structural stability of a barycentric model of the form $H(s)=\frac{N(s)}{D(s)}$, being $N(s)$ and $D(s)$ rational functions sharing the same asymptotically stable poles, is guaranteed by placing the zeros of $D(s)$ over the open left-half complex plane, i.e., by requiring this function to be minimum phase. The stable zero placement is practically achieved by constraining the denominator $D(s)$ to be a Strictly Positive Real (SPR) function~\cite{brogliato2007dissipative}, a condition that is easily achieved by enforcing suitable algebraic constraints on the free model coefficients. Since any SPR function is also minimum phase, the approach guarantees that all the zeros of the barycentric denominator have strictly negative real part, so that the stability of the model holds as a by-product. To formulate a \texttt{AAA} algorithm with certified stability, we expand on this pre-existing approach, generalizing the class of denominator functions that are admissible when an asymptotically stable model is required.
 
To do so, we first show that 
under mild assumptions, a barycentric \texttt{AAA} model $\hat H(s) = \frac{N(s)}{D(s)}$ is asymptotically stable if and only if the rational denominator $D(s)$ has only asymptotically stable zeros; second, exploiting the results of~\cite{ProofMP_ASPR}, we show that under the working conditions such requirement is equivalent to the enforcement of a relaxed positive realness condition on $D(s)$, which can be verified regardless of the location of the model's barycentric nodes (support points) chosen to define the model structure. Third, we show how the proposed condition is algebraically related to the free coefficients of the \texttt{AAA} model, i.e. its barycentric weights. Based on such characterization, we propose to replace the standard least-squares optimization stage of the \texttt{AAA} with a constrained counterpart, which admits only solutions associated with asymptotically stable models. Since the resulting optimization problem is non-convex with respect to the optimization variables, we propose an ad-hoc convex relaxation scheme that can be applied to efficiently solve for the model unknowns via semidefinite programming. Finally, we suggest how to incorporate the proposed optimization routine into the \texttt{AAA} iteration; the resulting algorithm, named stable \texttt{AAA} or \texttt{stabAAA}, retains the beneficial features of its unconstrained counterpart while guaranteeing that the resulting ROM is asymptotically stable. The effectiveness of the new approach is evaluated numerically on a number of practical MOR benchmarks selected from the mechanical, electrical, and acoustic domains, and compared with competing methods in terms of modeling performance.

The paper is organized as follows. In Sec.~\ref{sec:NotationAndProblemStatement} we introduce our notation and we provide a statement for the problem that we are solving. In Sec.~\ref{sec:AAA} we resume the necessary background material regarding the \texttt{AAA} algorithm, and the Loewner Framework, upon which it was designed. Sec.~\ref{sec:stabAAA} contains our main results in terms of stability characterization, enforcement, and formulation of the desired stable \texttt{AAA} algorithm. Sec.~\ref{sec:num} presents the results obtained by performing MOR of stable dynamical systems with the proposed approach and with other state-of-the-art related methods. Finally, conclusions are drawn in Sec.~\ref{sec:conc}.

\section{Notation and problem statement}
\label{sec:NotationAndProblemStatement}

In the following, $j=\sqrt{-1}$ and $s=\sigma + j\omega\in \IC$ denote respectively the imaginary unit and the Laplace variable. Lowercase and uppercase italic letters will be reserved for scalar numbers or functions of the $s$ variable respectively (e.g. $z=H(s): \IC\rightarrow \IC$). Given a complex number $z\in\IC$, $z^*$ represents its conjugate. The operators $\texttt{Re}\{\cdot\}$ and $\texttt{Im}\{\cdot\}$ return respectively the real and the imaginary parts of their arguments. Lowercase and uppercase blackboard bold letters will be used for vectors and matrices respectively (e.g. $\bx\in\IR^n$ and $\bX\in \IR^{n\times n}$). The matrices $\bX^T$ and $\bX^\star$ are the transpose and the Hermitian transpose of $\bX$. The matrix $\bI_n$ is the identity matrix of size $n$. The expression $\bJ\succ\bK$ ($\bJ\succeq\bK$) means that the matrix $\bJ-\bK$ is positive (semi)definite.

Let us consider a Single-Input Single-Output (SISO) LTI dynamical system $\Sig$ in descriptor form
\begin{equation}\label{def_lin_sys}
	\Sig: \begin{cases}
		\bE\,\dot{\boldsymbol{\zeta}}(t) & \!\!\!\!= \bA\, \boldsymbol{\zeta}(t) + \bB\, u(t) \\
		y(t)  & \!\!\!\!= \bC\, \boldsymbol{\zeta}(t) + \bD\, u(t)
	\end{cases},
\end{equation} 
with $\boldsymbol{\zeta}(t) \in \IR^{N}$ as the state variable, $u(t) \in \IR$ as the control input, and $y(t) \in \IR$ is the observed output. Consequently, we have $\bA,\bE \in \IC^{N \times N}$, $\bB \in \IC^{N \times 1}$, $\bC \in \IC^{1 \times N}$, $\bD \in \IR$. When the matrix $\bE$ is non-singular, system~\eqref{def_lin_sys} is equivalent to a state space realization for the system $\bSigma$. When $\bE=\bI_N$, we will use the notation $(\bA, \bB,\bC,\bD)$ to denote the state space realization resulting from~\eqref{def_lin_sys}. The transfer function associated to~\eqref{def_lin_sys} is given by $H(s): \IC \rightarrow \IC$
\begin{equation}\label{eqn:Htrf}
	H(s) = \bC (s \bE - \bA)^{-1} \bB +\bD,
\end{equation}
and is assumed to verify condition $H^*(s)=H(s^*)$, so that the impulse response of $\bSigma$ is real-valued. In particular, the poles of $H(s)$ are either real or appear in complex conjugate pairs. We assume that all the poles of $H(s)$ have strictly negative real part so that $\bSigma$ is asymptotically stable.

In our setting, the system $\bSigma$ is known only in terms of measurements of its transfer function, retrieved over a set of discrete sampling points located along the positive imaginary axis:
\begin{equation}\label{eq:realData}
   j\lambda_v \in \Gamma=\{j\lambda_1,\hdots, j\lambda_V\} \subset j\IR^+, \quad h_v=H(j\lambda_v).
\end{equation}
Starting from this data, our objective is to build an asymptotically stable reduced-order LTI system $\hat \Sig$ of (unknown) size $n\ll N$ with transfer function $\hat H(s)$
fulfilling the following conditions
\begin{equation}
    \hat H^*(s)=\hat H(s^*), \quad |\hat H(j\lambda_v)-h_v|<\epsilon \quad \forall j\lambda_v \in \Gamma,
\end{equation}
where $\epsilon>0$ is a user-defined error tolerance. Due to the stability requirements on $\hat \bSigma$, all the poles of $\hat H(s)$ must have a strictly negative real part. Therefore, our problem is to design a rational approximation strategy which concurrently 
\begin{enumerate}
     \item guarantees a user-prescribed error bound,\label{req:1}
     \item needs no information on the order $n$ of the resulting model function,\label{req:2}
    \item returns a real-valued and asymptotically stable model function $\hat H(s)$.\label{req:3}
\end{enumerate}

\section{Transfer function interpolation, barycentric forms, and the \texttt{AAA} algorithm}
\label{sec:AAA}

We start by recalling two established data-driven methods for rational approximation, i.e., the Loewner framework (LF) in \cite{MA07}, and the \texttt{AAA} algorithm in \cite{NST18}. We will be emphasizing mostly the latter approach, which is at the core of our current contribution. We discuss the model structure, realization, and enforcement of realness of \texttt{AAA} models, together with stability preservation and how it was handled in previous works in the literature.

\subsection{The Loewner framework}
Originally introduced in \cite{MA07}, the \texttt{LF} provides an elegant and direct answer to the generalized realization problem. This framework is based on processing the frequency domain measurements defined in~\eqref{eq:realData}. The interpolation problem is formulated according to a partitioning of the available samples into two disjoint subsets, by defining
\begin{equation}
   \Gamma_{\textrm R}=\{j\eta_1,\hdots, j\eta_{V_{\textrm R}}\}, \quad  \Gamma_{\textrm L}=\{j\mu_1,\hdots, j\mu_{V_{\textrm L}}\},\quad \Gamma_{\textrm R}\cup\Gamma_{\textrm L}=\Gamma, \quad \Gamma_{\textrm R}\cap\Gamma_{\textrm L}=\emptyset
\end{equation}
and the sets of \emph{right} and \emph{left} data pairs
\begin{align}\label{data_Loew}
	\begin{split}
		{\textrm{right data}}&: \cD_{\textrm R} = \{\left(j\eta_l,H(j\eta_l)\right), ~j\eta_l \in \Gamma_{\textrm R},~l=1,\ldots,V_{\textrm R}\},~{\textrm {and}}, \\
		{\textrm{left data}}&: \cD_{\textrm L}=\{\left(j\mu_i,H(j\mu_i)\right), ~j\mu_i \in \Gamma_{\textrm L}, ~i=1,\ldots,V_{\textrm L}\}.
	\end{split}
\end{align}
The approach seeks a rational function
$\hat{H}(s)$, such that the following interpolation conditions hold:
\begin{equation} \label{interp_cond}
	\hat{H}(j\mu_i)=H(j \mu_i),\quad\hat{H}(j\eta_l)=H(j\eta_l),\quad l=1,\ldots,V_{\textrm R},\quad i=1,\ldots,V_{\textrm L}.
\end{equation}
The Loewner matrix $\IL \in\IC^{V_{\textrm L} \times V_{\textrm R}}$ and the shifted Loewner matrix $\sIL \in\IC^{V_{\textrm L} \times V_{\textrm R}}$ play a crucial role in the \texttt{LF}, and are explicitly expressed as follows
\begin{equation} \label{Loew_mat}
	\IL_{(i,l)}=\frac{H(j \mu_i)-H(j\eta_l)}{j\mu_i-j\eta_l}, \ \ \sIL_{(i,l)}=
	\frac{j \mu_i H(j \mu_i)-j\eta_l H(j\eta_l)}{j\mu_i-j\eta_l},
\end{equation}
while the data vectors $\IV \in \IC^{V_\textrm{L}},\ \IW\trans \in \IC^{V_\textrm{R}}$ are given by $\left(\IV\right)_{i}= H(j \mu_i), \ \  \left(\IW\right)_{l} = H(j\eta_l),~\text{for}~i=1,\ldots,V_{\textrm L}$ and $l=1,\ldots, V_{\textrm R}$. The unprocessed Loewner surrogate model (provided that $V_{\textrm L}=V_{\textrm R}$) is composed of the matrices
\begin{align}\label{Loew_doub_sided}
	\hat{\bE}=-\IL,~~ \hat{\bA}=-\sIL,~~ \hat{\bB}=\IV,~~ \hat{\bC}=\IW,
\end{align}
and if the pencil $(\IL,\sIL)$ is regular, then the function $\hat{H}(s)$ satisfying the interpolation conditions in \cref{interp_cond} can be explicitly computed in terms of the matrices in \cref{Loew_doub_sided}, as $\hat{H}(s) = 	\hat{\bC} (s	\hat{\bE} - \hat{\bA})^{-1} \hat{\bB}$.

In practice, the pencil $(\sIL,\,\IL)$ is often singular. Hence, the Loewner model in \cref{Loew_doub_sided} needs to be post-processed. By computing two singular value decompositions (SVDs) of appropriately-arranged Loewner matrices, the dominant features can be extracted and inherent redundancies in the data can be removed. Possibly advantageous data splitting choices were proposed in \cite{morKarGA19a}. Finally, it is to be noted that the choice of well-suited interpolation points is crucial for the \texttt{LF}. Different choices were evaluated in \cite{morKarGA19a}, while a greedy strategy was proposed in \cite{cherifi2022greedy}. Additionally, issues such as stability preservation/enforcement, and passivity preservation, were tackled in the \texttt{LF} in \cite{morGosA16,morGosPA21a}, for the former, and in \cite{antoulas2005new,morBenGV20}, for the latter. For exhaustive tutorial contributions on \texttt{LF} for LTI systems, we refer the reader to \cite{ALI17} and \cite{morKarGA19a}. 

\subsection{The Adaptive Antoulas-Anderson (\texttt{AAA}) algorithm}

The \texttt{AAA} algorithm, as originally introduced in \cite{NST18}, is a greedy rational approximation scheme that iteratively increases the complexity of the rational approximant in order to meet a user-prescribed error tolerance over the available data samples. The main steps of the algorithm, which is intimately connected to the \texttt{LF}, are:
\begin{enumerate}
	\item Express the fitted rational approximants in a barycentric representation, which is a numerically stable way of representing rational functions \cite{berrut2004barycentric}.
	
	\item Select the next interpolation (support) points via a greedy scheme by enforcing interpolation at the point where the (absolute or relative) error at the previous step is maximal.

	\item Compute the other variables (the so-called barycentric weights) in order to enforce least squares approximation on the non-interpolated data.
\end{enumerate}	
The \texttt{AAA} algorithm was further extended and developed in recent years, including applications to nonlinear eigenvalue problems \cite{lietaert2022automatic}, MOR of parameterized linear dynamical systems \cite{CS20}, MOR of linear systems with quadratic outputs \cite{gosea2022data}, rational approximation of periodic functions \cite{baddoo2021aaatrig}, representation of conformal maps \cite{gopal2019representation}, rational approximation of matrix-valued functions \cite{GG20,morAumBGetal23,osti2005602}, signal processing with trigonometric rational functions \cite{wilber2022data}, approximations on continua sets \cite{driscoll2023aaa}, or reconstruction of sparse exponential sums \cite{derevianko2023esprit}. 

\subsubsection{\texttt{AAA} model structure and realization}

In the case of frequency response data, the rational approximant $\hat{H}(s)$ computed by the original \texttt{AAA} algorithm in \cite{NST18} is based, at any given step of the iteration, upon the barycentric form:
\begin{equation}\label{eq:AAAmodel}
	\hat{H}(s)=  \frac{\sum_{i=1}^k \displaystyle \frac{w_i h_i}{s-z_i}}{\sum_{i=1}^k \displaystyle \frac{w_i}{s-z_i}}=\frac{N(s)}{D(s)},
\end{equation}
where the $z_i \in \IC$ are the barycentric nodes (or support points), the $w_i$'s are the barycentric weights, the $h_i$'s are the barycentric values, and $N(s),D(s)$ are both rational functions. It is to be noted that the following interpolation conditions are satisfied, by construction, at the support points (provided that $w_i \neq 0$), i.e.: $\hat{H}(z_i) = h_i$ for all $1\leq i \leq k$. Typically, the values $h_i$'s are evaluations of the unknown transfer function $H(s)$ to be modeled at the $z_i$ locations, i.e., $h_i := H(z_i)$. In this case, we say that $\hat{H}(s)$ is a \textit{rational interpolant}, since it interpolates the original transfer function $H(s)$, at chosen support points $z_i$'s, i.e., $\hat{H}(z_i) = H(z_i)$.

For model order reduction applications, the support points are commonly placed on the imaginary axis, so that $z_i=j\lambda_i \in \Gamma$ as defined in~\eqref{eq:realData}. In this setting, for the requirement $\hat H^*(s)= \hat H(s^*)$ to hold, it is necessary to consider the following modified barycentric expression: 
\begin{align}\label{eq:AAAReal}
     \hat{H}(s) &=\frac{\sum_{i=1}^k \left( \frac{h_i w_i}{s -j \lambda_i} + \frac{(h_i w_i)^*}{s +j\lambda_i} \right) } {\sum_{i=1}^k \left(\frac{w_i}{s -j\lambda_i} + \frac{w_i^*}{s +j\lambda_i}\right)} = \frac{\sum_{i=1}^k N_i(s)}{\sum_{i=1}^k D_i(s)},
\end{align}
where $N_i(s) = \frac{h_i w_i}{s -j \lambda_i} + \frac{(h_i w_i)^*}{s +j\lambda_i}$ and $D_i(s) = \frac{w_i}{s -j\lambda_i} + \frac{w_i^*}{s +j\lambda_i}$.
This structure was previously used in \cite{morGosPA21a} and \cite{valera2021aaa}, and also in other subsequent publications.

The following Lemma provides a closed-form expression for realizing the system $\hat \Sig$ in descriptor form when its transfer function is structured as in~\eqref{eq:AAAReal}. See \Cref{sec:App1} for the full proof of the Lemma.
\begin{lemma}
\label{lemm:3.1}
As adapted from \cite{gosea2020rational}, a realization for the transfer function $\hat{H}(s)$ in \cref{eq:AAAReal} is given by
\begin{align}\label{ROM_SISO_improper}
	\begin{split}
		\tilde{\bA} &= \bLambda - \tilde{\bB} \tilde{\bR}  \in \IC^{(2k+1) \times (2k+1)}, \ \ \tilde{\bB} = \begin{bmatrix} w_1 & w_1^* & \cdots & w_k & w_k^* & 1 \end{bmatrix}\trans \in \IC^{(2k+1) \times 1}, \\ \tilde{\bC} &=  \begin{bmatrix} h_1& h_1^* & \cdots &  h_k & h_k^* & 0
		\end{bmatrix} \in \IC^{1 \times (2k+1)}, \  \tilde{\bE} = \operatorname{diag}(1,1,\ldots,1,0) \in \IC^{(2k+1) \times (2k+1)},\\
  \bLambda &= \operatorname{diag}(j\lambda_1,-j\lambda_1,\ldots,j\lambda_k,-j\lambda_k,1) \in \IC^{(k+1) \times (k+1)}, \  \tilde{\bR} = \begin{bmatrix}
      1 & 1 & \cdots & 1
  \end{bmatrix} \in \IC^{1 \times (2k+1)},
	\end{split}
\end{align}
i.e., its transfer function, given by $\tilde{H}(s) = \tilde{\bC} (s \tilde{\bE} - \tilde{\bA})^{-1} \tilde{\bB}$, satisfies the equality $\tilde{H}(s) = 	\hat{H}(s) $.
\end{lemma}

In some applications, especially in electrical engineering, it is sometimes desired that the input to state vector $\bB$ (for the SISO case) contain ones only, with the row vector $\bC$ storing the residues of the transfer function. This can always be achieved by constructing an equivalent realization, as detailed in~\Cref{rem:remApp2}. Equally important, the topic of purely real-valued realizations arises next. For more details on the exact procedure of how to obtain such realizations, we refer the reader to \Cref{rem:remApp3}.

\subsubsection{The real-valued \texttt{AAA} algorithm}
In the following, we summarize the \texttt{AAA} algorithm when the model structure is constrained to obey equation~\eqref{eq:AAAReal}. Assuming the availability of a data set in form~\eqref{eq:realData}, the \texttt{AAA} algorithm builds the final model~\eqref{eq:AAAReal} following a greedy iterative procedure, driven by an error minimization criterion. This iteration is aimed at returning a model which satisfies
\begin{equation}
    \max_{v=1,\hdots V} | \hat{H}(j\lambda_v)-h_v|\leq \epsilon, \quad j\lambda_v\in \Gamma
\end{equation}
where $\epsilon>0$ is an admissible approximation error tolerance. By defining the iteration index $\ell$, and denoting with $\hat{H}_{\ell}(s)$ the model obtained at the $\ell$-th iteration, the algorithm is initialized by defining for $\ell = 0$
\begin{equation}
    \hat{H}_0(s)=V^{-1}\sum_{v=1}^V h_v. \qquad \Gamma^{(0)}=\Gamma.
\end{equation}
The $\ell$-th iteration starts by determining the sample point $j\lambda_\ell$ belonging to the current test set $\Gamma^{(\ell-1)}$ over which the maximum error value is attained,
\begin{equation}\label{eq:maxAAAerror}
    j\lambda_\ell = \arg \max_{j\lambda \in \Gamma^{(\ell-1)}} |H(j\lambda)-\hat{H}_{\ell-1}(j\lambda)|,
\end{equation}
and updating the definition of the model and of the test set consequently
\begin{equation}\label{eq:modelUpdate}
     \hat{H}_\ell(s):= \frac{\sum_{i=1}^\ell \left( \frac{h_i^{(\ell)} w_i^{(\ell)}}{s -j \lambda_i} + \frac{(h_i^{(\ell)} w_i^{(\ell)})^*}{s +j\lambda_i} \right) } {\sum_{i=1}^\ell\left(\frac{w_i^{(\ell)}}{s -j\lambda_i} + \frac{(w_i^{(\ell)})^*}{s +j\lambda_i}\right)} \qquad \Gamma^{(\ell)}:= \Gamma^{(\ell-1)}\setminus \{j\lambda_\ell\}.
\end{equation}
The complex weights $w_i^{(\ell)}$ represent the current model unknowns, that are found by enforcing the linearized approximation of $\hat{H}_\ell(s)$ against $H(s)$ over the updated test set $\Gamma^{(\ell)}$:
\begin{align}\label{eqn:LSfit_Loew1}
			\begin{split}
				&\displaystyle \sum_{i=1}^\ell\left(\frac{w_i^{(\ell)}}{s -j\lambda_i} + \frac{(w_i^{(\ell)})^*}{s +j\lambda_i}\right) H(s) \approx \sum_{i=1}^\ell\left( \frac{h_i^{(\ell)} w_i^{(\ell)}}{s -j \lambda_i} + \frac{(h_i^{(\ell)} w_i^{(\ell)})^*}{s +j\lambda_i} \right)  \\
    & \Leftrightarrow \ \  \Big{(} \sum_{i=1}^\ell\frac{H(s) - h_i^{(\ell)}}{s - j\lambda_i}  w_i^{(\ell)}\Big{)} + \Big{(} \sum_{i=1}^\ell\frac{H(s) - (h_i^{(\ell)})^*}{s + j\lambda_i}  (w_i^{(\ell)})^*\Big{)} := E^{(\ell)}(s) \approx 0  \quad \forall s\in \Gamma^{(\ell)}.
			\end{split}
		\end{align}
  By collecting the unknowns in the vector $\bw^{(\ell)} = [w_1^{(\ell)},(w_1^{(\ell)})^*,\ldots,w_\ell^{(\ell)},(w_\ell^{(\ell)})^*]^T$, condition~\eqref{eqn:LSfit_Loew1} is enforced by solving the following homogeneous least-squares problem
  \begin{equation}\label{eq:LoewnerMinimization}
      \bw_\textrm{opt}^{{(\ell)}}= \arg \min_{|| \bw^{(\ell)}||_2=1}  ||\IL^{(\ell)} \bw^{(\ell)}||^2_2 \qquad \IL^{(\ell)} \in \mathbb{C}^{(V-\ell)\times 2\ell}
  \end{equation}
  where $\IL^{(\ell)}$ is the Loewner matrix and the constraint $||\bw^{(\ell)}||_2=1$ is introduced to rule out the trivial zero solution. The minimizer $\bw_\textrm{opt}^{(\ell)}$ is recognized as the right singular vector of $\IL^{(\ell)}$ associated to its least singular value. The algorithm stops whenever the maximum error (computed according to the cost function in~\eqref{eq:maxAAAerror}) is below the threshold tolerance $\epsilon$.
  
From a numerical standpoint, computing the minimizer of problem~\eqref{eq:LoewnerMinimization} via SVD does not guarantee that the coefficients  $w_i^{(\ell)}$ are obtained as exact complex conjugate pairs. Since this condition is necessary to ensure that $\hat H_{\ell}(s)$ is real-valued, the optimization problem~\eqref{eq:LoewnerMinimization} must be reformulated accordingly. This can be done by changing the optimization variables and rewriting the optimization problem in terms of real quantities. First, note that the approximation condition~\eqref{eqn:LSfit_Loew1} is equivalent to
\begin{equation}\label{eq:realimagapprox}
      \texttt{Re}(E^{(\ell)}(s))\approx 0, \quad \texttt{Im}(E^{(\ell)}(s))\approx 0, \qquad \forall s \in \Gamma^{(\ell)}
  \end{equation}
when both are are enforced in a least-squares sense; then, by observing the conditions in~\eqref{eq:realimagapprox} are linear in the unknowns \mbox{$\alpha_i^{(\ell)}=\texttt{Re}(w_i^{(\ell)})$} and $\beta_i^{(\ell)}=\texttt{Im}(w_i^{(\ell)})$, we can define the following vector of unknowns, $\bx^{(\ell)}=[ \alpha_1^{(\ell)},\beta_1^{(\ell)}\hdots \alpha_\ell^{(\ell)},\beta_\ell^{(\ell)} ]^T$, and enforce~\eqref{eq:realimagapprox} by solving the following homogeneous least-squares problem
\begin{equation}\label{eq:LoewnerMinimizationReal}
      \bx_\textrm{opt}^{{(\ell)}}= \arg \min_{|| \bx^{(\ell)}||_2=1}  ||\IL_{\textrm{real}}^{(\ell)} \bx^{(\ell)}||_2, \ \ \text{with} \ \ \IL_{\textrm{real}}^{(\ell)} \in \mathbb{R}^{2(V-\ell)\times 2\ell},
\end{equation}
where now, all the involved quantities are real-valued and the row dimension of the matrix $\IL_{\textrm{real}}^{(\ell)}$ is twice that of $\IL^{(\ell)}$ since the approximation conditions on real and imaginary components are considered separately. The formal definition of the real-valued \enquote{quasi-Loewner} matrix is explicitly provided in Appendix~\ref{app:MatrixStructure}. Algorithm~\ref{al:aaa} summarizes the above procedure in a pseudocode format.
\begin{algorithm}[tb]
	\caption{The modified \texttt{AAA} algorithm with enforced realness.}
	\label{al:aaa}
	\begin{algorithmic}[1]
		\Require A (discrete) set of sample points $j\lambda_v, v=1,\hdots V \in \Gamma \subset j\mathbb{R}^+$ and corresponding target transfer function values, $h_v=H(j\lambda_v)$. An error tolerance $\epsilon>0$.
		\Ensure A rational approximant $\hat H(s)$ displayed in a barycentric form.
		\State Initialize $\ell=1$, $\Gamma^{(0)} \gets \Gamma$, and $\hat{H}_{0}(s) \gets V^{-1}\sum_{v=1}^{V} h_v$.
		\While{$\max_{j\lambda \in \Gamma^{(\ell-1)}} |H(j\lambda)-\hat{H}_{\ell-1}(j\lambda)| > \epsilon$}
        \State Set $j\lambda_{\ell}=\arg \max_{j\lambda \in \Gamma^{(\ell-1)}} |H(j\lambda)-\hat{H}_{\ell-1}(j\lambda)|$ and define $\hat H_{\ell}(s)$ and $\Gamma^{(\ell)}$ according to~\eqref{eq:modelUpdate}
  \State Solve for $\bx_\textrm{opt}^{{(\ell)}}= \arg \min_{|| \bx^{(\ell)}||_2=1}  ||\IL_{\textrm{real}}^{(\ell)} \bx^{(\ell)}||_2$, with $\bx_\textrm{opt}^{(\ell)}=[ \alpha_{1,\textrm{opt}}^{(\ell)},\beta_{1,\textrm{opt}}^{(\ell)}\hdots \alpha_{\ell,\textrm{opt}}^{(\ell)},\beta_{\ell,\textrm{opt}}^{(\ell)} ]$.
   \State $w_i^{(\ell)}\gets\alpha_{i,\textrm{opt}}^{(\ell)}+j\beta_{i,\textrm{opt}}^{(\ell)}$.
		\State $\ell \gets \ell+1$.
		\EndWhile
  \State return $\hat H_{\ell-1}(s)$
	\end{algorithmic}
\end{algorithm}

\subsection{Stability of \texttt{AAA} models}

The \texttt{LF} and the \texttt{AAA} algorithm have proven to be reliable and efficient methods for the computation of interpolatory rational functions. The accuracy of the interpolant can be defined arbitrarily, making the methods suited for reduced-order modeling of dynamical systems. Also, the greedy approach pursued by the \texttt{AAA} algorithm allows not to a priori specify the order of the ROM $\hat\Sig$, thus automating the modeling procedure.
However, by construction, the algorithm is not guaranteed to return a rational approximation with Hurwitz poles, so that the resulting ROM $\hat \Sig$ may exhibit unstable dynamics even when the underlying system $\Sig$ is asymptotically stable. This condition is likely to occur when the algorithm is not required to fit exactly the underlying data, i.e. when the error threshold $\epsilon$ is not extremely aggressive, a common scenario in engineering applications. When the ROM $\hat \Sig$ lacks stability, it becomes basically useless for a number of practical applications, in particular for any time-domain simulation.

The problem of enforcing stable models using \texttt{LF} or \texttt{AAA} has been addressed in previous works. Enforcing stability in terms of \texttt{LF} is typically realized with a post-processing step incorporating pole flipping, truncation of unstable poles, or a combination of both \cite{morGosA16,kergus2020data,carrera2021improving,osti2005602,valera2021aaa,aumann2023practical}. In \cite{Cool22}, stability-preserving reduction techniques in the Loewner framework were proposed for the efficient time-domain simulation of dynamical systems with damping treatments.

Similar techniques as in \cite{morGosA16} have been included in the \texttt{AAA} framework in \cite{morGosPA21a}, where an optimization problem is solved to find the closest stable approximation to the originally unstable \texttt{AAA} interpolant. 
A different trajectory is suggested in \cite{deckers2022time}, where the original barycentric formulation of the \texttt{AAA} approximant is replaced by a different rational basis. In this representation, the unstable poles of a \texttt{AAA} approximation can be truncated. The updated weights are then recomputed by solving a least squares problem. Similarly, in \cite{osti2005602}, the unstable poles are replaced with their stable reflection, and the corresponding weights are adapted accordingly. Alternatively, it is also possible to pre-define the poles of the surrogate model, i.e., such that no post-processing step is required.
It was shown in \cite{aumann2023practical} that in order to place the poles of the \texttt{AAA} approximant, one needs to solve a linear system of equations involving a Cauchy matrix. 

It is to be noted that, although effective in general, none of the above approaches is guaranteed to return the optimal stable \texttt{AAA} model: for any obtained ROM with order $n$, there may exist a stable model that approximates the available data with smaller residual error. In this view, designing a \texttt{AAA} algorithm with guaranteed stability is still an open research field. In the following section, we present our novel approach to characterize algebraically the asymptotic stability of models generated with the \texttt{AAA} algorithm, as well as to enforce stability during the model generation.

\section{Stability characterization and enforcement for \texttt{AAA} models}
\label{sec:stabAAA}

When dealing with rational functions in barycentric form, such as~\eqref{eq:AAAmodel} or~\eqref{eq:AAAReal}, a strict relation holds between the poles of $\hat H(s)$ and the zeros of the denominator function $D(s)$. As will be more precisely formalized next, the poles of $\hat H(s)$  can be found only at locations where $D(s)$ exhibits a zero. Based on this observation, we propose to guarantee the stability of $\hat H(s)$ by placing such zeros over the open left-half complex plane while solving for the model unknowns. 

Inspired by previous works related to the stability enforcement of rational barycentric models in the Sanathanan-Koerner approximation framework~\cite{BraddeBernstein,ZancoPositive}, we achieve the goal by studying the positive realness properties of the denominator function $D(s)$. We show that this approach provides an algebraic characterization for the stability of $\hat H(s)$ in terms of the barycentric weights $w_i$. We exploit such characterization to formulate a set of constraints on the model unknowns that, under non-restrictive conditions, are necessary and sufficient for $\hat H(s)$ being asymptotically stable.

\subsection{Stability characterization}
We start by recalling some basic definitions.
\begin{definition}[Relative Degree]
Let
\begin{equation}\label{eq:tfPoly}
    F(s)=\frac{f_1(s)}{f_2(s)}
\end{equation}
be a scalar rational function of $s$ and let $f_1$ and $f_2$ be coprime polynomials with $\emph{deg}~f_1(s)=d_1$ and $\emph{deg}~f_2(s)=d_2$. Then the relative degree of $F(s)$ is $r=d_2-d_1$, while the McMillan degree of $F(s)$ is $m = \max{(d_1,d_2)}$.
\end{definition}

\begin{definition}[Minimum Phase~\cite{ilchmann1993non}]\label{def:MinimumPhase}
Let $f_1(s)$ and $f_2(s)$ in~\eqref{eq:tfPoly} be coprime polynomials. Then $F(s)$ is said to be Minimum Phase (MP) if
\begin{equation}\label{eq:noUnstableZeros}
   f_1(s)\neq0 \quad \forall s: \emph{\texttt{Re}}(s)\geq 0.
\end{equation}
Additionally, a state space system $(\bA,\bB,\bC,\bD)$ with $F(s)=\bD+\bC(s\bI_n-\bA)^{-1}\bB$ is said to be MP if it is stabilizable, detectable~\cite{ACA05} and~\eqref{eq:noUnstableZeros} holds.
\end{definition}
With these definitions, we can already state the following
\begin{proposition}\label{prop:MPStability}
    Let 
    \begin{equation}{\label{eq:barycentric}}
        \hat{H}(s)=  \frac{\sum_{i=1}^k \displaystyle \frac{w_i h_i}{s-z_i}}{\sum_{i=1}^k \displaystyle \frac{w_i}{s-z_i}}=\frac{N(s)}{D(s)}=\frac{\frac{p(s)}{l(s)}}{\frac{q(s)}{l(s)}} = \frac{p(s)}{q(s)}
    \end{equation}
    with $w_i\neq 0$, $N(s):=\sum_{i=1}^k \displaystyle \frac{w_i h_i}{s-z_i}$, $D(s):=\sum_{i=1}^k \displaystyle \frac{w_i}{s-z_i}$, and with $p(s),q(s),l(s)$ polynomials in $s$.
    Then, the following statements hold true
    \begin{itemize}
    \item Whenever $D(s)$ is Minimum Phase, $\hat{H}(s)$ is asymptotically stable.
    \item Assuming that the polynomials $p(s),q(s)$ share no roots in the closed right-half complex plane, then the poles of $\hat{H}(s)$ are asymptotically stable if and only if $D(s)$ is Minimum Phase       
    \end{itemize}
\end{proposition}
\begin{proof}
    First, notice that when $w_i\neq 0$, the function $D(s)={\frac{q(s)}{l(s)}}$ has exactly $k$ poles, hence $l(s)$ is of degree $k$ and is coprime with $q(s)$. Then, according to Definition~\ref{def:MinimumPhase}, when all the roots of $q(s)$ have a strictly negative real part, $D(s)$ is MP. Since the poles of $\hat H(s)$ are the roots of $q(s)$ (or a subset of them when zero-poles cancellations occur), MP-ness of $D(s)$ suffices to guarantee that $\hat H(s)$ is asymptotically stable. This proves the first statement of the proposition.

    When $p(s)$ and $q(s)$ share no roots in the closed right-half complex plane, then $\hat H(s)$ is asymptotically stable if and only if all the roots of $q(s)$ have a strictly negative real part. Hence, in this case, the MP-ness of $D(s)$ is equivalent to the stability of $\hat H(s)$, and the second statement is proved.

\end{proof}
The following represent classical definitions of Positive Real (PR) and Strictly Positive Real (SPR) scalar rational functions.
\begin{definition}[Positive Real (PR) function~\cite{brogliato2007dissipative}]
    A rational function $F(s)$ is positive real if
    \begin{enumerate}
        \item $F(s)$ has no poles in $\emph{\texttt{Re}}(s)>0$
        \item $F(s)$ is real for all positive real $s$
        \item $\emph{\texttt{Re}}(F(s))\geq0$ for all $\emph{\texttt{Re}}(s)>0$
    \end{enumerate}
\end{definition}
\begin{definition}[Strictly Positive Real (SPR) function~\cite{brogliato2007dissipative}]\label{def:SPRfrequency}
    A rational function $F(s)\in \IC$ that is not identically zero for all $s$ is strictly positive real if $F(s-\tau)$ is PR for some $\tau>0$.
\end{definition}

A well-known property of SPR transfer functions~\cite{brogliato2007dissipative} is that if $F(s)$ is SPR, so is its inverse $1/F(s)$. Since the SPR condition requires the regularity of the transfer function over the closed right-half complex plane, the property implies that both the poles and the zeros of an SISO SPR transfer function must be asymptotically stable; therefore any SISO SPR transfer function is also MP.

The following result, which was proven in~\cite{taylor1974strictly}, represents a special case of the well-known Positive Real Lemma for SISO transfer functions and provides the link between the state space representation of an LTI system and the positive realness of its transfer function.
\begin{theorem}\label{th:LYPlemma}
    Consider a SISO state space $(\bA,\bB,\bC,\bD)$ with $F(s)=\bD+\bC(s\bI_n-\bA)^{-1}\bB$. Suppose that $\det(s\bI_n-\bA)$ has only zeros in the open left half complex plane. Suppose $(\bA,\bB)$ is controllable and let $c>0$, $\bL=\bL^T\succ 0$ be an arbitrary real symmetric positive definite matrix. Then a real vector $\bq$ and a real matrix $\bP=\bP^T\succ 0$ satisfying
    \begin{equation}
        \begin{cases}
            \bP\bA + \bA^T \bP = -\bq \bq^T -c \bL \\
            \bP \bB - \bC^T=\sqrt{2\bD}\bq,
        \end{cases}
    \end{equation}
    exist if and only if $F(s)$ is SPR and $c$ is sufficiently small.
\end{theorem}
For state-space systems without feedthrough term, we will make use of the following equivalent statement based on a Linear Matrix Inequality (LMI).

\begin{theorem}
    Let $F(s)=\bC(s\bI_n-\bA)^{-1}\bB$, with  $(\bA,\bB,\bC)$ fulfilling the assumptions of \Cref{th:LYPlemma}. Then the conditions
    \begin{align}
        &\bP\bA + \bA^T \bP\prec 0 \label{eq:KYP1}\\
        &\bP\bB=\bC^T \label{eq:KYP2}
    \end{align}
    are fulfilled with $\bP=\bP^T\succ 0$ if and only if $F(s)$ is SPR.
\end{theorem}

\begin{proof}
    This comes directly from \Cref{{th:LYPlemma}} by plugging in $\bD=0$ into the second equality, and also by relaxing the first equality since its right-hand side is an arbitrary negative definite matrix.
\end{proof}

The link between the stability of the \texttt{AAA} model structure~\eqref{eq:AAAReal}, MP, and SPR transfer functions is established by the following theorem, proven in~\cite{ProofMP_ASPR} and reported here for the SISO case.
\begin{theorem}\label{th:ASPR}
Any strictly proper, minimum phase SISO transfer function $F(s)=\bC (s \bI_n - \bA)^{-1} \bB$ with $\bC\bB>0$  can be made SPR via constant output feedback, so that
\begin{equation}\label{eq:ASPR}
    \exists g \in \mathbb{R}:
    G(s)=\frac{F(s)}{1+gF(s)} = \bC(s \bI_n - \bA+g\bB \bC)^{-1}\bB \quad \text{is SPR}
\end{equation}
\end{theorem}
The above Theorem states that every strictly proper, SISO MP transfer function with $\bC \bB>0$ can generate an SPR transfer function when closed in feedback with a constant output gain. We highlight that the condition $\bC\bB>0$ is necessary to get SPR-ness of the closed-loop state space \mbox{$(\bA-g\bB\bC,\bB,\bC)$}, since in order to fulfill~\eqref{eq:KYP2}, for this system it must hold $\bC\bB=\bB^T\bP\bB$ with $\bP$ positive definite. Transfer functions satisfying the requirements of \Cref{th:ASPR} are referred to in the literature as \emph{Almost Strictly Positive Real} transfer functions, see e.g.~\cite{ProofMP_ASPR}.

\Cref{th:ASPR} can be used to check whether the denominator $D(s)$ in~\eqref{eq:AAAReal} is minimum phase. Thus, it provides a tool to characterize the asymptotic stability of the \texttt{AAA} transfer function $\hat{H}(s)$. The following result comes naturally.
\begin{lemma}\label{stabilityLemma}
    Let 
    \begin{equation}\label{eq:AAAlemma}
        \hat{H}(s) =\frac{\sum_{i=1}^k \left( \frac{h_i w_i}{s -j \lambda_i} + \frac{(h_i w_i)^*}{s +j\lambda_i} \right) } {\sum_{i=1}^k \left(\frac{w_i}{s -j\lambda_i} + \frac{w_i^*}{s +j\lambda_i}\right)} = \frac{N(s)}{D(s)}=\frac{\frac{p(s)}{l(s)}}{\frac{q(s)}{l(s)}} = \frac{p(s)}{q(s)}, \ \ \text{with} \ D(s)= \Cden (s \bI_{2k} - \Aden)^{-1} \Bden,
    \end{equation}
where the realization $(\Aden, \Bden, \Cden)$ is of size $2k$; assume that $w_i\neq 0$ and that $\hat H(s)$ has $2k-1$ poles. Then $\hat{H}(s)$ is asymptotically stable if and only if
    \begin{enumerate}
        \item $\Cden\Bden>0$ and $\exists g\in \mathbb{R}$ such that 
        \begin{equation}
             G^+(s)=
           \frac{D(s)}{1+gD(s)} = \Cden(s \bI_{2k} - \Aden+g\Bden \Cden)^{-1}\Bden
        \end{equation}
        is SPR or, \label{cond:stabplus}
         \item $\Cden\Bden<0$ and $\exists g\in \mathbb{R}$ such that 
        \begin{equation}
             G^-(s)=
           \frac{-D(s)}{1-gD(s)} = -\Cden(s \bI_{2k} - \Aden-g\Bden \Cden)^{-1}\Bden
        \end{equation}
        is SPR. \label{cond:stabminus}
    \end{enumerate}
\end{lemma} 
\begin{proof}
With reference to the definition
\begin{equation}
        D(s)=\sum_{i=1}^k \left(\frac{w_i}{s -j\lambda_i} + \frac{w_i^*}{s +j\lambda_i}\right)=\frac{q(s)}{l(s)}=\frac{\sum_{m=0}^{2k-1} a_m s^m}{\sum_{n=0}^{2k} b_n s^n}, 
\end{equation}
we recall that when $\hat H(s)$ has $2k-1$ poles, these poles coincide with the roots of the polynomial $q(s)$. Hence this polynomial must have $2k-1$ roots, implying $a_{2k-1}\neq 0$. As proved in~\cite{NST18}, when $w_i\neq 0$, $\hat H(s)$ has no poles at $\pm j\lambda_i$; hence $q(s)$ and $l(s)$ are coprime and any realization of $D(s)$ of size $2k$ is minimal. Noticing that $a_{2k-1}=\sum_{i=1}^k(w_i+w_i^*)$ is the sum of the residues of $D(s)$ and that the quantity $\Cden\Bden$ is realization invariant, we see that the equality $a_{2k-1}=\Cden\Bden$ holds for every minimal realization since it holds for the Gilbert realization~\cite{ACA05}. Hence, when $\hat H(s)$ has $2k-1$ poles, it must hold that $\Cden\Bden\neq 0$.

Now, when $\hat H(s)$ has $2k-1$ poles, $p(s)$ and $q(s)$ are coprime, and, according to Proposition~\ref{prop:MPStability}, $\hat H(s)$ is asymptotically stable if and only if $D(s)$ is MP.

Consider the case in which $\Cden\Bden>0$; due to \Cref{th:ASPR}, $D(s)$ is MP only if there exists $g$ such that $G^+(s)$ is SPR, so Condition~\ref{cond:stabplus} is necessary for the stability of $\hat H(s)$. To prove sufficiency, notice that the constant output feedback cannot induce pole-zero cancellations in $G^+(s)$ for $g\in \mathbb{R}$. In fact, such cancellation would occur if both the conditions $G(s_z)=0$ and $1+gG(s_z)=0$ were verified concurrently, which is impossible. Hence, the zeros of $G^+(s)$ coincide with the zeros of $D(s)$. If $G^+(s)$ is SPR, then these zeros have strictly negative real parts, $D(s)$ is MP and $\hat{H}(s)$ is asymptotically stable. 

The same arguments can be applied to prove the statement of the theorem when $\Cden\Bden<0$, by simply considering that $D(s)$ is MP if and only if $-D(s)$ is MP.
\end{proof}

Lemma~\ref{stabilityLemma} establishes an explicit algebraic link between the state space realization of the denominator function $D(s)$ and the stability of the \texttt{AAA} model $\hat{H}(s)$. In the following section, we show how to exploit this link in order to constrain the \texttt{AAA} algorithm to generate structurally stable models.
Figure~\ref{fig:block-diagram} provides a graphical illustration of the relations between the various properties that were discussed above.

\begin{figure}
\centering
    \begin{tikzpicture}
 
\draw (-3.8,-0.5) rectangle ++(10.5,5.5);
  
  \node at (1.5,4.5) {$F(s)=\bC(s\bI_{n}-\bA)^{-1}\bB$, with $(\bA,\bB,\bC)$ minimal and $\bC\bB>0$};

  \draw[red](0.75, 1.5) ellipse (2.7 and 1.7); 

  \node at (0.75,0.3) {{\color{red}$\textrm{MP}\equiv\textrm{ASPR}$}};

  \draw[blue] (1.5, 2.5) ellipse (1.5 and 1.2);

   \node at (2.2,3.2) {{\color{blue}$\textrm{PR}$}};

    \draw[orange] (1.25, 2.15) ellipse (1 and 0.5);

   \node at (1.25, 2.1) {{\color{orange}$\textrm{SPR}$}};
\end{tikzpicture}
    \caption{Graphical illustration of the various relations between Minimum-Phase (MP), Almost Strictly Positive Real (ASPR), Positive Real (PR), and Strictly Positive Real (SPR) functions. The intersection between PR and the set of Not-MP contains PR functions having zeros on $j\mathbb{R}$.}
    \label{fig:block-diagram}
\end{figure}
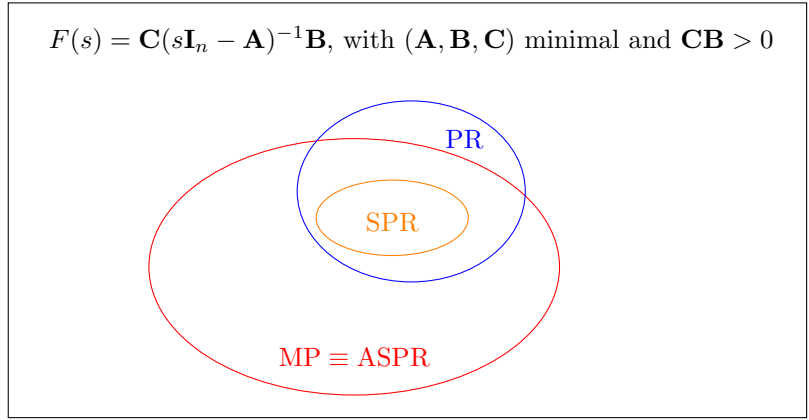

\subsection{Exact stability enforcement}
Consider the denominator function $D(s)$ in~\eqref{eq:AAAlemma}
\begin{equation}
    D(s)=\sum_{i=1}^k \left(\frac{w_i}{s -j\lambda_i} + \frac{w_i^*}{s +j\lambda_i}\right) = \Cden (s \bI_{2k} - \Aden)^{-1} \Bden,
\end{equation}
with $w_i\neq0$ and $j\lambda_i\in j\mathbb{R}^+$. We define the following minimal state space realization for $D(s)$
\begin{equation}
    \Aden=\textrm{blkdiag}[\Aden_1,\hdots,\Aden_k],\quad \Bden=[\Bden_1,\hdots, \Bden_k]^T, \quad \Cden=[ \alpha_1,\beta_1,\hdots ,\alpha_k,\beta_k ],
\end{equation}
with
\begin{equation}
    \Aden_i=\begin{bmatrix}
        0&\lambda_i\\-\lambda_i&0
    \end{bmatrix}, \quad
    \Bden_i=\begin{bmatrix}
        2&0
    \end{bmatrix}, \quad w_i=\alpha_i+j\beta_i, \quad \forall i=1,\hdots, k.
\end{equation}

We focus on the optimization problem~\eqref{eq:LoewnerMinimizationReal} solved at each iteration of the real-valued \texttt{AAA} algorithm. For readability, we take into account a single iteration $\ell=k\geq 1$ and we drop the iteration index $\ell$, rewriting the problem as
\begin{equation}\label{eq:OptStep0}
      \bx_\textrm{opt}= \arg \min_{|| \bx||_2=1}  ||\IL_{\textrm{real}} \bx||_2=\arg \min_{\Vert \Cden||_2=1}  ||\IL_{\textrm{real}} \Cden^T\Vert^2_2,
\end{equation}
exploiting the fact that $\Cden=[ \alpha_1,\beta_1,\hdots ,\alpha_k,\beta_k ]=\bx^T$.  The problem can be equivalently rewritten in terms of $\Cden$ dropping the non-convex norm constraint as
\begin{equation}\label{eq:OptStep1}
     \Cden_\textrm{opt}= \arg\min_{\Cden}\Vert\IL_{\textrm{real}} (\Cden^T-\bx_\textrm{opt})\Vert^2_2,
\end{equation}
where we assume that the ambiguity on the sign of $\bx_\textrm{opt}$ is removed by requiring $\bx_\textrm{opt}^T\Bden>0$.
When an asymptotically stable approximant $\hat H(s)$ with $2k-1$ poles is required, Lemma~\ref{stabilityLemma} states that the \texttt{AAA} model obtained by solving~\eqref{eq:OptStep1} is asymptotically stable if and only if the resulting denominator function $D(s)$ can generate a SPR transfer function under constant output feedback. Therefore, we introduce the required constraints into problem~\eqref{eq:OptStep1}, in order to guarantee that any of its feasible solutions will lead to a model with the desired stability property.

First, we observe that the length of the vector defining the cost function of~\eqref{eq:OptStep1} can be reduced by performing the state space transformation
\begin{equation}
     D(s)\Leftrightarrow (\tilde \Aden,\tilde \Bden,\tilde \Cden)=(\bT^{-1}\Aden \bT,\bT^{-1}\Bden,\Cden \bT), \quad \IL_{\textrm{real}}=\bU\bSigma\bV^T, \quad \bT=\bV\bSigma,
\end{equation}
where $\IL_{\textrm{real}}=\bU\bSigma\bV^T$ is the ``thin'' SVD of $\IL_{\textrm{real}}$, which is assumed to be full column rank. Notice that this condition is practically verified in any realistic model order reduction process. Using this state space realization, the cost function of~\eqref{eq:OptStep1} is rewritten in terms of $\tilde{\Cden}=\Cden\bT$ as
\begin{equation}
   \Vert\IL_{\textrm{real}} (\Cden^T-\bx_\textrm{opt})\Vert^2_2= \Vert\bSigma\bV^T(\Cden^T-\bx_\textrm{opt})\Vert_2^2= 
   \Vert \tilde{\Cden}^T-\bar{\bx}\Vert_2^2,
\end{equation}
where $\bar{\bx}=\bSigma \bV^T \bx_\textrm{opt}$, leading to the equivalent problem
\begin{align}\label{eq:OptStep2}
    \min_{\Cden, \tilde \Cden}&=\Vert \tilde{\Cden}^T-\bar{\bx}\Vert_2^2 \\
    &\text{subject to}: \nonumber\\
      &\tilde{\Cden}=\Cden\bT \nonumber.
\end{align}
Notice that the vector defining the cost function of~\eqref{eq:OptStep2} has size $2k$, while the vector entering~\eqref{eq:OptStep1} has size $2(V-k)$. Since the number of data points $V$ is much larger than the number of unknowns $2k$ in every practical application, the variable transformation substantially reduces the size of the problem.

To enforce model stability, we restrict the feasible set of~\eqref{eq:OptStep2} to the set of vectors $\tilde{\Cden}$ for which Condition~\eqref{cond:stabplus} of Lemma~\ref{stabilityLemma} can be verified, and we consider feasible any vector $\tilde{\Cden}$ such that the function $G^+(s)\Leftrightarrow(\tilde{\Aden}-g\tilde{\Bden} \tilde{\Cden},\tilde{\Bden},\tilde{\Cden})$ fulfills the SPR conditions~\eqref{eq:KYP1},~\eqref{eq:KYP2} with suitable $g\in \mathbb{R}$. Thus, according to~\eqref{eq:ASPR}, the constrained optimization problem thus reads
\begin{align}
 \min_{\Cden,\tilde \Cden,\bQ,g} &\Vert \tilde{\Cden}^T-\bar{\bx}\Vert_2^2\label{eq:OptStep3}\\
    &\text{subject to:} \nonumber \\
    &\bQ=\bQ^T\succ 0,\\ 
    & \tilde{\Aden}^T\bQ+\bQ\tilde{\Aden}-g(\tilde{\Bden} \tilde{\Cden})^T\bQ-g\bQ\tilde{\Bden} \tilde{\Cden}\prec0,  \nonumber \\&\bQ\tilde{\Bden}=\tilde{\Cden}^T, \quad \tilde{\Cden}=\Cden\bT. \nonumber
\end{align}
Since the constraints involved in~\eqref{eq:OptStep3} are necessary and sufficient to verify the stability of model structure~\eqref{eq:AAAReal} when it is assumed to have $2k-1$ poles, the optimal solution $\tilde{\Cden}_\textrm{opt}^T=\bx_\textrm{opt}$ is feasible whenever the unconstrained \texttt{AAA} solution $\bx_\textrm{opt}$ leads to an asymptotically stable model. When this is not the case, the stability constraints are active and the optimal solution $\tilde{\Cden}^T_\textrm{opt}$ will not coincide with $\bar{\bx}$.
Eliminating the equality constraints from~\eqref{eq:OptStep3}, the problem is equivalently rewritten as 
\begin{align}
    \min_{\bQ,g}&\Vert \bQ\tilde{\Bden}-\bar{\bx}\Vert_2^2\label{eq:OptStep4} \\
    &\text{subject to:} \nonumber\\
     &\bQ=\bQ^T\succ 0, \nonumber\\ 
    & \tilde \Aden^T\bQ+\bQ \tilde \Aden-2g \bQ\tilde \Bden \tilde \Bden^T \bQ\prec0 \nonumber.
\end{align}

\subsection{Stability enforcement based on convex optimization}

Unfortunately, the optimization problem~\eqref{eq:OptStep4}  is non-convex, since it requires the fulfillment of a nonlinear matrix inequality in variables $\bQ$ and $g$. Although a number of methods based on convex relaxations have been proposed to tackle this kind of problems (see e.g.~\cite{convexOverbounding}), none of them guarantees the recovery of the global solution. On the other hand, treating the problem as is, employing for example global optimization techniques, would be practically feasible only in case very few optimization variables are involved. In the following, we show that problem~\eqref{eq:OptStep4} can be relaxed to a convex problem that preserves the exactness of the model stability constraints, at the price of a slight modification of the target cost function.

\subsubsection{Problem relaxation}

To obtain a convex relaxation of~\eqref{eq:OptStep4}, we start by noticing that the non-linear matrix inequality involved in the problem can be turned into a linear one by applying a congruence transformation. Defining $\bY=\bY^T= \bQ^{-1}$ we have
\begin{equation}
    \begin{aligned}
    \tilde \Aden^T\bQ+\bQ \tilde \Aden-2g \bQ\tilde \Bden \tilde \Bden^T \bQ\prec0 \quad &\iff \quad\bY(\tilde \Aden^T\bQ+\bQ \tilde \Aden-2g \bQ\tilde \Bden \tilde \Bden^T \bQ)\bY \prec 0.\\
    & \iff \quad \bY\tilde \Aden^T+\tilde \Aden \bY-2g\tilde \Bden \tilde \Bden^T\prec 0.
    \end{aligned}
\end{equation}
Second, recalling that scaling the objective function with an arbitrary positive constant $\eta>0$ does not change the minimizer of the problem, we rewrite~\eqref{eq:OptStep4} equivalently as
\begin{align}
    \min_{\bQ,\bY,g} & \: \eta \Vert \bQ\tilde{\Bden}-\bar{\bx}\Vert_2^2\label{eq:OptStep5} \\
    &\text{subject to:} \nonumber \\
     &\bQ=\bQ^T \succ0, \nonumber\\ 
    & \bY\tilde \Aden^T+\tilde \Aden \bY-2g\tilde \Bden \tilde \Bden^T\prec0, \nonumber\\
    &\bY=\bQ^{-1} \nonumber.
\end{align}
Although the matrix inequalities that guarantee the model stability are linear in the unknown matrices $\bY$,$\bQ$, the problem is still non-convex due to the presence of the constraint $\bY=\bQ^{-1}$. Applying an inverse Schur complement, the problem is turned into epigraph form as
\begin{align}
    \min_{\bQ,\bY,g,r} & \:r\label{eq:OptStep6} \\
    &\text{subject to:} \nonumber \\
    &\bQ=\bQ^T \succ0,\nonumber\\ 
    & \bY\tilde \Aden^T+\tilde \Aden \bY-2g\tilde \Bden \tilde \Bden^T\prec0, \nonumber\\
    &\bY=\bQ^{-1}, \nonumber\\
    &\begin{bmatrix}
         r & (\bQ\tilde{\Bden}-\bar{\bx})^T \\ \bQ\tilde{\Bden}-\bar{\bx} & \frac{1}{\eta}\bI_{2k}
     \end{bmatrix} \succeq 0 \label{eq:schur2Breplaced}.
\end{align}
The last matrix inequality can be equivalently restated in terms of $\bY$ by applying the following congruence transformation
\begin{equation}\label{eq:congruence}
    \begin{bmatrix}
        1&\bold{0}^T\\\bold{0}&\bY
    \end{bmatrix}
    \begin{bmatrix}
         r & (\bQ\tilde{\Bden}-\bar{\bx})^T \\ \bQ\tilde{\Bden}-\bar{\bx} & \frac{1}{\eta}\bI_{2k}
     \end{bmatrix}
     \begin{bmatrix}
        1&\bold{0}^T\\\bold{0}&\bY \end{bmatrix}=\begin{bmatrix}
         r & (\tilde{\Bden}-\bY\bar{\bx})^T \\ \tilde{\Bden}-\bY\bar{\bx} & \frac{1}{\eta}\bY\bY\succeq0  
     \end{bmatrix}\succeq0.
\end{equation}
Replacing~\eqref{eq:schur2Breplaced} with~\eqref{eq:congruence} and introducing the additional constraint $\bY\succ 0$ allows to eliminate the variable $\bQ$. We then proceed by performing a convex relaxation of~\eqref{eq:congruence} relying on the following property, known as Young's relation~\cite{caverly2019lmi}
\begin{equation}
    \bR,\bW \in \mathbb{R}^{n\times m}, \bS=\bS^T\succ0 \in \mathbb{R}^{n\times n}\quad \implies \quad \bR^T\bW + \bW^T\bR \preceq \bR^T \bS^{-1}\bR+\bW^T\bS\bW,
\end{equation}
which, setting $\bR=\bY$, $\bW=\bI_{2k}$ and $\bS=\eta \bI_{2k}$ leads to
\begin{equation}\label{eq:Young}
    \frac{1}{\eta}\bY \bY \succeq 2\bY - \bI_{2k} \eta.
\end{equation}
Due to~\eqref{eq:Young} and~\eqref{eq:congruence}, any feasible solution of the following problem
\begin{align}
    \min_{\bY,g,r} & \:r\label{eq:OptStep7} \\
    &\text{subject to:} \nonumber \\
    &\bY \succ 0, \nonumber\\ 
    & \bY\tilde \Aden^T+\tilde \Aden \bY-2g\tilde \Bden \tilde \Bden^T\prec0, \nonumber\\
   & \begin{bmatrix}
         r & (\tilde{\Bden}-\bY\bar{\bx})^T \\ \tilde{\Bden}-\bY\bar{\bx} & 2\bY- \bI_{2k} \eta
     \end{bmatrix} \succeq 0. \nonumber
\end{align}
is a feasible solution of~\eqref{eq:OptStep5}, and thus of~\eqref{eq:OptStep4}, with $\bQ=\bY^{-1}$.
Finally, since we assumed that $\eta$ is an arbitrary positive constant, the following implication
\begin{equation}
    \bY \succ 0, \quad \begin{bmatrix}
         r &  (\tilde{\Bden}-\bY\bar{\bx})^T \\  \tilde{\Bden}-\bY\bar{\bx} & \bY
     \end{bmatrix} \succeq 0 \quad \implies \quad \begin{bmatrix}
         r &  (\tilde{\Bden}-\bY\bar{\bx})^T \\  \tilde{\Bden}-\bY\bar{\bx}& 2\bY- \bI_{2k} \eta
     \end{bmatrix} \succeq 0
\end{equation}
always holds for sufficiently small $\eta$. Thus, we obtain the following relaxation of~\eqref{eq:OptStep4}
\begin{align}
    \min_{\bY,g,r} & \:r\label{eq:OptStep8} \\
    &\text{subject to:} \nonumber \\
    &\bY \succ 0,\nonumber\\ 
    & \bY\tilde \Aden^T+\tilde \Aden \bY-2g\tilde \Bden \tilde \Bden^T\prec0, \nonumber\\
   & \begin{bmatrix}
         r & (\tilde{\Bden}-\bY\bar{\bx})^T \\ \tilde{\Bden}-\bY\bar{\bx} & \bY
     \end{bmatrix} \succeq 0. \nonumber
\end{align}
which is convex in the decision variables and can be solved via standard convex optimization methods. Once its solution is retrieved, the optimal coefficient vector $\Cden$ which defines the \texttt{AAA} model can be recovered exploiting the relations $\bY=\bQ^{-1}$, $\bQ\tilde\Bden=\tilde \Cden^T$, $\tilde \Cden = \Cden \bT$. The corresponding solution will guarantee the asymptotic stability of the resulting ROM.
 
\subsubsection{The cost function of the relaxed convex problem}\label{sec:costFunc}

The proposed relaxation scheme leads to a problem that is equivalent to~\eqref{eq:OptStep4} up to a modification of its cost function. This can be seen by taking the Schur complement of the last matrix inequality of~\eqref{eq:OptStep8}, obtaining
\begin{equation}
    r\geq (\tilde{\Bden}-\bY\bar{\bx})^T \bQ (\tilde{\Bden}-\bY\bar{\bx})=\Vert \bQ^{\frac{1}{2}}\tilde{\Bden} -  \bQ^{-\frac{1}{2}}\bar{\bx}\Vert_2^2 = \Vert\bQ^{-\frac{1}{2}} (\bQ\tilde{\Bden}-\bar{\bx})\Vert_2^2,
\end{equation}
which implies that the epigraph formulation of~\eqref{eq:OptStep8} minimizes the cost $\Vert\bQ^{-\frac{1}{2}} (\bQ\tilde{\Bden}-\bar{\bx})\Vert_2^2$. We conclude that the problem solved by the proposed relaxation is equivalent to the following 
\begin{align}\label{eq:relaxedWithQ}
    \min_{\bQ,g}&\Vert \bQ^{-\frac{1}{2}} (\bQ\tilde{\Bden}-\bar{\bx})\Vert_2^2 \\
    &\text{subject to:} \nonumber\\
     &\bQ=\bQ^T\succ 0, \nonumber\\ 
    & \tilde \Aden^T\bQ+\bQ \tilde \Aden-2g \bQ\tilde \Bden \tilde \Bden^T \bQ\prec0 \nonumber,
\end{align}
since every feasible solution $\bY$ of~\eqref{eq:OptStep8} returns a matrix $\bQ=\bY^{-1}$ that fulfils the constraints of~\eqref{eq:relaxedWithQ}, and the two problems minimize the same cost function.

Comparing the exact non-convex stability enforcement problem \eqref{eq:OptStep4} with~\eqref{eq:relaxedWithQ}, we see that the two differ only by the factor $\bQ^{-\frac{1}{2}}$ entering the latter cost function. We highlight in first place that, in both cases, when the constraints are not active, the solution of the problems is the same matrix $\bQ_\textrm{opt}$ such that $\bQ_\textrm{opt}\Bden=\bar{\bx}$, which in turn implies $\Cden_\textrm{opt}^{T}=\bx_\textrm{opt}$ in terms of the model optimization variables. Therefore, in case the unconstrained optimal solution $\bx_\textrm{opt}$ should lead to an asymptotically stable model, such solution would be recovered by solving the proposed relaxed problem~\eqref{eq:OptStep8}. 

On the other hand, when the stability constraints are active, then the weighting factor $\bQ^{-\frac{1}{2}}$ in the cost function of~\eqref{eq:relaxedWithQ} will effectively modify the problem solution. However, as will be shown in the following, neither the reference cost function of~\eqref{eq:OptStep4}, nor the cost function of the unconstrained \texttt{AAA} algorithm, represent exactly the $2$-norm of the residual fitting error of the model $\hat{H}(s)$ with respect to the available sample points. Therefore, a more careful assessment of the role of the weighting factor introduced by~\eqref{eq:relaxedWithQ} is in order.

Let us consider that the \texttt{AAA} algorithm performs rational approximation by interpolating the available data at given support points, and solving a least-squares problem to minimize the deviation of $\hat H(s)$ against the remaining data samples. The optimization stage would formally require the solution of the following optimization problem (dropping any iteration index)
\begin{equation}\label{eq:problemExact}
    \bw_\textrm{opt}=\arg\min_{\norm{\bw}_2=1} \Vert \bfe \Vert_2, \quad \bfe=[\hat H(s_1)-h_1,\hdots, \hat H(s_L)-h_L]^T,  \quad L=\textrm{card}(\Gamma),
\end{equation}
which is however non-convex in the model coefficients and known to exhibit poor local solutions. Therefore, instead of solving~\eqref{eq:problemExact}, the algorithm enforces the linearized approximation condition~\eqref{eqn:LSfit_Loew1} in least-squares sense, solving problem~\eqref{eq:LoewnerMinimization}. When written in terms of the exact error vector $\bfe$, this problem reads
\begin{equation}\label{eq:linearizedApproxOpt}
    \bw_\textrm{opt}=\arg\min_{\norm{\bw}_2=1} \Vert \IL \bw\Vert_2 = \arg\min_{\norm{\bw}_2=1} \Vert \bDelta \bfe \Vert_2=\arg\min_{\norm{\bw}_2=1} \bfe^{\star}\bDelta^{\star}\bDelta\bfe \quad \bDelta=\textrm{diag}[D(s_1),\hdots, D(s_L)].
\end{equation}
The above shows that the linearization process introduces a weight (or bias) in the desired cost function which depends on the squared absolute value of the unknown denominator $D(s)$, evaluated over $s\in\Gamma$. Before solving for $\bw$, this weighting factor is unknown. The removal of such bias is the main idea behind a number of iterative algorithms such as the Vector Fitting~\cite{VF} and the similar Sanathanan-Koerner~\cite{SK} iteration. For all of them, the core machinery is the solution of a least-squares problem of the kind of~\eqref{eq:linearizedApproxOpt} with suitable modifications.

When the real-valued \texttt{AAA} structure is enforced, as in model structure~\eqref{eq:AAAReal}, and computations are made with real algebra, problem~\eqref{eq:LoewnerMinimization} is replaced by~\eqref{eq:LoewnerMinimizationReal}, that can be written as
\begin{equation}
    \bx_{\textrm{opt},1}= \arg \min_{\norm{\bx}_2=1} \Vert\IL_{\textrm{real}} \bx\Vert_2=\arg \min_{\norm{\bx}_2=1} \norm{\begin{bmatrix} \texttt{Re}(\bDelta\bfe)\\ \texttt{Im}(\bDelta\bfe)\end{bmatrix}}^2_2,\quad \IL_\textrm{real}\bx_{\textrm{opt},1}=\bfe_1^\bw,
\end{equation}
where $\bfe_1^\bw$ is the optimal linearized error vector resulting from the problem solution. The cost function of~\eqref{eq:OptStep4}, can be rewritten using the above notation as
\begin{equation}
    \Vert \IL_{\textrm{real}} \Cden^T - \bfe_1^\bw\Vert_2^2=\Vert \bfe_2^\bw - \bfe_1^\bw\Vert_2^2= (\bfe_2^\bw - \bfe_1^\bw)^T(\bfe_2^\bw - \bfe_1^\bw), \quad \IL_{\textrm{real}} \Cden^T=\bfe_2^\bw 
\end{equation}
being $\bfe_2^\bw$ the linearized residual error vector of the \texttt{AAA} model defined by the coefficients $\Cden$. The cost function that is minimized while solving~\eqref{eq:OptStep8} introduces a weighting factor in the optimization. Such cost function can be written in terms of $(\bfe_2^\bw - \bfe_1^\bw)$ as
\begin{equation}
    (\bfe_2^\bw - \bfe_1^\bw)^T\bU \bY \bU^T(\bfe_2^\bw - \bfe_1^\bw), \quad \bY=\bQ^{-1},\quad \IL_{\textrm{real}}=\bU \bSigma \bV^T,
\end{equation}
and is thus observed to weight the deviation between the residual error of the two models through matrix $\bU \bY \bU^T$. Since such deviation does not have a straightforward interpretation in terms of the actual residual fitting error of the model defined by $\Cden$, it is difficult to state the effect of this weighting on the overall modeling performance of the algorithm. We can state that solving both the original \texttt{AAA} optimization~\eqref{eq:OptStep0} or the proposed constrained relaxation~\eqref{eq:OptStep8} leads to the minimization of a weighted approximation error. In both cases, the introduced weighting factor is unknown before obtaining the weighted optimal solution. This makes it impossible to state theoretically which of the two approaches will return the model with the smallest residual error $\bfe$. From the practical standpoint, the effectiveness of the proposed convex stability enforcement scheme is confirmed by the numerical experiments reported in Sec.~\ref{sec:num}.

\subsection{A real-valued \texttt{AAA} algorithm with guaranteed stability}

In this section, we suggest how to embed the convex program~\eqref{eq:OptStep8} into the real-valued \texttt{AAA} algorithm in order to generate structurally stable reduced-order transfer functions. In principle, this could be done by simply modifying Step $4$ in Algorithm~\ref{al:aaa}, by solving~\eqref{eq:OptStep8} for the optimal matrix $\bY_\textrm{opt}$ in place of~\eqref{eq:LoewnerMinimizationReal}, and recovering the required coefficient vector $\bx_\textrm{opt}^{(\ell)}$ via prescribed variable transformations. However, the resulting  algorithm at each iteration
should solve a semi-definite program in place of simply computing the SVD of $\IL_{\textrm{real}}^{(\ell)}$, while repeatedly introducing an additional weighting factor in the standard \texttt{AAA} cost function, due to the proposed convex relaxation approach.
For these reasons, we found that a more efficient and effective approach is to solve~\eqref{eq:OptStep8} only once after the last iteration, and only if the model returned by Algorithm~\ref{al:aaa} is found to be unstable. In case the resulting stable model does not meet the required maximum error tolerance, we simply perform more \texttt{AAA} iterations and increase the number of poles of $\hat H(s)$. A pseudocode for the proposed approach, which we call stable \texttt{AAA}, or \texttt{stabAAA}, is given in Algorithm~\ref{al:aaa2}.

\subsubsection{Can we always attain the desired accuracy?}

We highlight that there are at least two relevant scenarios of practical interest in which a stable model $\hat H (s)$ fulfilling the condition $\max_{j\lambda_v\in \Gamma}\vert h_v-\hat H(j\lambda_v) \vert < \epsilon$ for arbitrarily small $\epsilon$ may not exist. These scenarios occur when
\begin{enumerate}
\item The target transfer function $ H(s)$ exhibits one or more unstable poles or, in general, unstable dynamics. Since a closed-form expression for $ H(s)$ may be not available, it is not generally possible to infer from the available data if such a scenario occurs. If it does, a constrained stable rational approximation $\hat  H(s)$ for $ H(s)$ which satisfies a maximum error bound with arbitrarily small threshold $\epsilon$ may not exist.
\item The data samples may include noise and/or some non-causal components, both of which are incompatible with an arbitrarily accurate rational approximation with constrained stability. Such spurious components can come from direct measurement or even from the discretization error of Partial Differential Equations providing a first-principle physical description of the underlying dynamics. Also, models of material properties that are not causal (e.g., as fitted from measurements using non-causal models). These difficulties are well-known and have been extensively studied~\cite{triverio2007stability}. Detection of non-causal data components is particularly challenging since any finite set of data samples is compatible with a causal model (e.g. by computing a rational interpolation with order equal to the number of data points). However, severe overfitting occurs when enforcing model causality hence stability when interpolating or approximating non-causal data. Efficient and reliable algorithms exist to detect such situation~\cite{jnl-2008-tadvp-dispersionrelations,6680676}.
\end{enumerate}
In the above scenarios, the approximation error of a constrained stable rational model saturates to a minimum allowed level even when arbitrarily increasing model order~\cite{triverio2007stability}. In this work, we assume that the error threshold imposed as a stopping criterion for model order selection is larger than the amount of non-causal components in the data.

\begin{algorithm}[tb]
	\caption{\texttt{stabAAA}:  \texttt{AAA} algorithm with stability enforcement.}
	\label{al:aaa2}
	\begin{algorithmic}[1]
		\Require A (discrete) set of sample points $j\lambda_v, v=1,\hdots V \in \Gamma \subset j\mathbb{R}^+$ and corresponding target transfer function values, $h_v= H(j\lambda_v)$. An error tolerance $\epsilon>0$. Tolerance decreasing factor $0<\theta<1$. A maximum iteration number $M_{max}$.
		\Ensure A stable rational approximant $\hat  H(s)$ of order $(2k,2k)$ displayed in a barycentric form.
		\State Initialize $\ell=1$, $\Gamma^{(0)} \gets \Gamma$, $\hat H_{0}(s) \gets V^{-1}\sum_{v=1}^{V} h_v$, $\epsilon_M=\epsilon$, $M=1$.
         \State Execute the \textbf{while} cycle of Algorithm~\ref{al:aaa} with error tolerance $\epsilon_M$. \label{step:while}
         \State Retrieve $\hat  H_{\ell-1}(s)$ and $\IL_{\textrm{real}}^{(\ell-1)}$ 
  \If{$\hat  H_{\ell-1}(s)$ is stable}
  \State return $\hat  H_{\ell-1}(s)$
  \Else
  \State Solve the optimization problem~\eqref{eq:OptStep8}, using $\IL_{\textrm{real}}^{(\ell-1)}$ and retrieve the optimal vector $\Cden_\textrm{opt}$
   \State $w_i^{(\ell-1)}\gets\alpha_{i,\textrm{opt}}^{(\ell-1)}+j\beta_{i,\textrm{opt}}^{(\ell-1)}$
   \If {$\max_{j\lambda \in \Gamma^{(\ell-1)}} |H(j\lambda)-\hat{H}_{\ell-1}(j\lambda)| < \epsilon$ \textbf{or} $M > M_{max}$} 
  \State return $\hat  H_{\ell-1}(s)$
  \Else
  \State $\epsilon_M \gets \theta\epsilon_M$.
  \State $M\gets M+1$
  \State Go to step~\ref{step:while}
  \EndIf
  \EndIf
	\end{algorithmic}
\end{algorithm}

\section{Numerical results and applications}\label{sec:num}

In this section, we apply \texttt{stabAAA} for generating asymptotically stable ROMs of large-scale systems starting from a set of frequency domain measurements. 

We select three test cases coming from electrical, mechanical, and acoustical domains. For each of them, Algorithm~\ref{al:aaa} (\texttt{AAA}) is first applied with an admissible approximation error $\epsilon$ such that the final model is detected as unstable. Then, using the same error threshold, we apply the proposed approach to generate a stable model and compare the results in terms of the accuracy of the final approximation. Once our stable model is available, we compare the proposed approach with the Vector Fitting iteration, with the recent approach~\cite{osti2005602} (\texttt{smiAAA}), and with the method presented in~\cite{valera2021aaa} (\texttt{AAA}$^2$); in this latter approach the unstable poles of an unconstrained \texttt{AAA} model are simply discarded and the model residues are optimized with respect to the remaining stable poles.  For fairness of comparison, we apply \texttt{VF} and \texttt{smiAAA} by imposing the same number of poles of the model obtained with the proposed approach. For \texttt{AAA}$^2$, the order is as obtained with Algorithm~\ref{al:aaa} less the number of unstable modes that are truncated by the algorithm.

Before applying any of the considered modeling approaches, each data set is normalized as follows
\begin{equation}
    f_{Max} = \max_{v=1,\hdots, V} \lvert \lambda_v/2\pi \lvert, \quad h_{Max}= \max_{v=1,\hdots, V} \lvert h_v \lvert, \quad h_v\leftarrow h_v/h_{Max}, \quad \lambda_v\leftarrow \lambda_v/f_{Max}, \quad v=1,\hdots, V.
\end{equation}
This allows standardization of the selected admissible approximation error $\epsilon$ irrespectively of the test case at hand.
For each test case, denoting with $X(s)$ any of the computed reduced-order transfer functions, we compute three error metrics
\begin{equation}   E_\infty^X=\max_{v=1,\hdots V} \lvert {X}(j\lambda_v)-h_v\lvert, \quad E_2^ X= \sqrt{\sum_{i=v}^V \lvert  {X}(j\lambda_v)-h_v \lvert^2},  \quad E_{RMS}^ X=\frac{E_2^ X}{\sqrt{V}}
\end{equation}
and we use them as accuracy indices in order to compare the performance of each modeling approach. Also, we plot the approximation obtained with our method against the sampled frequency response data, and the magnitude of the error along frequency for all the models we build. All the considered algorithms are implemented in MATLAB and executed on a workstation equipped with 32 GB of memory and a 3.3 GHz Intel i9-X7900 CPU. To solve problem~\eqref{eq:OptStep8} we use the YALMIP toolbox~\cite{Lofberg2004} and the solver MOSEK~\cite{mosek}. The codes used to perform the experiments can be found at \url{https://github.com/tomBradde/stabAAA}.

\subsection{A high-speed Printed Circuit Board link}
We consider a high-speed electrical interconnection between two Printed Circuit Boards (PCB), first presented in~\cite{preibisch2017exploring}. Our objective is to derive an ROM for the output admittance at one port of the interconnect. A high-fidelity characterization for this admittance function is computed via an electromagnetic field solver, as explained in~\cite{preibisch2017exploring}. The available dataset is composed of pairs $(j\lambda_v,h_v)=(j2\pi f_v,h_v), v=1,\hdots 497,  f_v \in [0.06,10]$~GHz. 

After data normalization, we apply Algorithm~\ref{al:aaa} to obtain a ROM $\hat H_{ns}(s)$ fulfilling the requirement $\epsilon=E^{\hat  H_{ns}}_\infty<0.001$. The algorithm hits the required approximation tolerance after $k=22$ iterations, returning a model with $E^{\hat  H_{ns}}_\infty= 1.84\times 10^{-4}$ and \mbox{$E^{\hat H_{ns}}_2=7.84\times 10^{-4}$}. The model has an unstable real pole, as shown in the top panel of Fig.~\ref{fig:FittingSchuster}. Therefore, we apply \texttt{stabAAA} using the same error threshold to generate a stable model $\hat H_{s}(s)$. The algorithm generates a model fulfilling the required error tolerance after $k=22$ iterations so that $\hat H_{s}(s)$ has the same complexity of $\hat H_{ns}(s)$. The optimization problem~\eqref{eq:OptStep8} is solved in $0.98$ s on our hardware. The approach effectively removes the unstable pole, returning a model with $E^{\hat H_{s}}_\infty= 1.84\times 10^{-4}$ and $E^{\hat H_{s}}_2=7.83\times 10^{-4}$. Figure~\ref{fig:FittingSchuster} (top panel) shows that all of the stable dominant poles identified by the standard \texttt{AAA} algorithm are recovered by the proposed constrained counterpart. In the middle panel of Fig.~\ref{fig:FittingSchuster} we compare the magnitude response of the reference data with that of $\hat H_{s}(s)$. The two curves are practically indistinguishable.
\begin{table}
\caption{Error metrics for all the considered test cases and generated models.}
\begin{center}
\begin{tabular}{|c|c|c|c|c|c|c|}
\hline
\multicolumn{2}{|c|}{} & \texttt{AAA} & \texttt{stabAAA} & \texttt{VF} & \texttt{AAA}$^2$ & \texttt{smiAAA}\\
\hline
\multirow{3}{*}{PCB}
& $E_2$ & $7.84\times 10^{-4}$ & $7.83\times 10^{-4}$ & $5.52\times 10^{-4}$ & $6.73\times 10^{-2}$ & $8.41\times 10^{-4}$\\ \cline{2-7}
& $E_{RMS}$ & $3.52\times 10^{-5}$ & $3.51\times 10^{-5}$ & $2.48\times 10^{-5}$ & $3.02\times 10^{-3}$ & $3.77\times 10^{-5}$\\ \cline{2-7}
& $E_\infty$ & $1.84\times 10^{-4}$ & $1.84\times 10^{-4}$ & $1.80\times 10^{-4}$ & $3.61\times 10^{-2}$ & $2.47\times 10^{-4}$ \\ \hline
\multirow{2}{*}{ISS}
& $E_2$ & $1.94\times 10^{-4}$ & $ 1.96\times 10^{-4}$ & $1.56\times 10^{-4}$ & $9.73\times 10^{-4}$ & $1.89\times 10^{-3}$\\ \cline{2-7}
& $E_{RMS}$ & $9.69\times 10^{-6}$ & $ 9.82\times 10^{-6}$ & $7.81\times 10^{-6}$ & $4.87\times 10^{-5}$ & $9.15\times 10^{-5}$\\ \cline{2-7}
& $E_\infty$ & $5.62\times 10^{-5}$ & $5.38\times 10^{-5}$ & $1.11\times 10^{-4}$ & $5.06\times 10^{-4}$ & $1.19\times 10^{-3}$\\ \hline   
\multirow{2}{*}{Absorber}
& $E_2$ & $1.74\times 10^{-8}$ & $ 2.55\times 10^{-5}$ & $1.65\times 10^{-5}$ & $2.37\times 10^{-4}$ & $1.14\times 10^{-4}$\\ \cline{2-7}
& $E_{RMS}$ & $5.8\times 10^{-10}$ & $ 8.48\times 10^{-7}$ & $5.55\times 10^{-7}$ & $7.9\times 10^{-6}$ & $3.8\times 10^{-6}$\\ \cline{2-7}
& $E_\infty$ & $1.77\times 10^{-9}$ & $1.92\times 10^{-6}$ & $2.68\times 10^{-6}$ & $4.42 \times 10^{-5}$ & $8.79\times 10^{-6}$ \\ \hline   
\end{tabular}
\end{center}
\label{tab:overall}
\end{table}
We then proceed to model this dataset with the competing methods. The magnitude of the error in the frequency domain is shown for all of the models in the bottom panel of Fig.~\ref{fig:FittingSchuster}, while a summary of the considered error metrics is given in Table~\ref{tab:overall}. The results show that for this test case, the proposed approach performs better than the other two derived from the \texttt{AAA} algorithm and is comparable with the \texttt{VF} iteration. We highlight that the \texttt{VF} model, by construction, does not interpolate the target data at any point, as can be seen in the bottom panel of Fig.~\ref{fig:FittingSchuster}; also, the number of poles required by the \texttt{VF} to achieve this performance had to be specified and was not automatically detected by the algorithm.

\begin{figure}
    \centering
     \includegraphics[width=0.8\columnwidth]{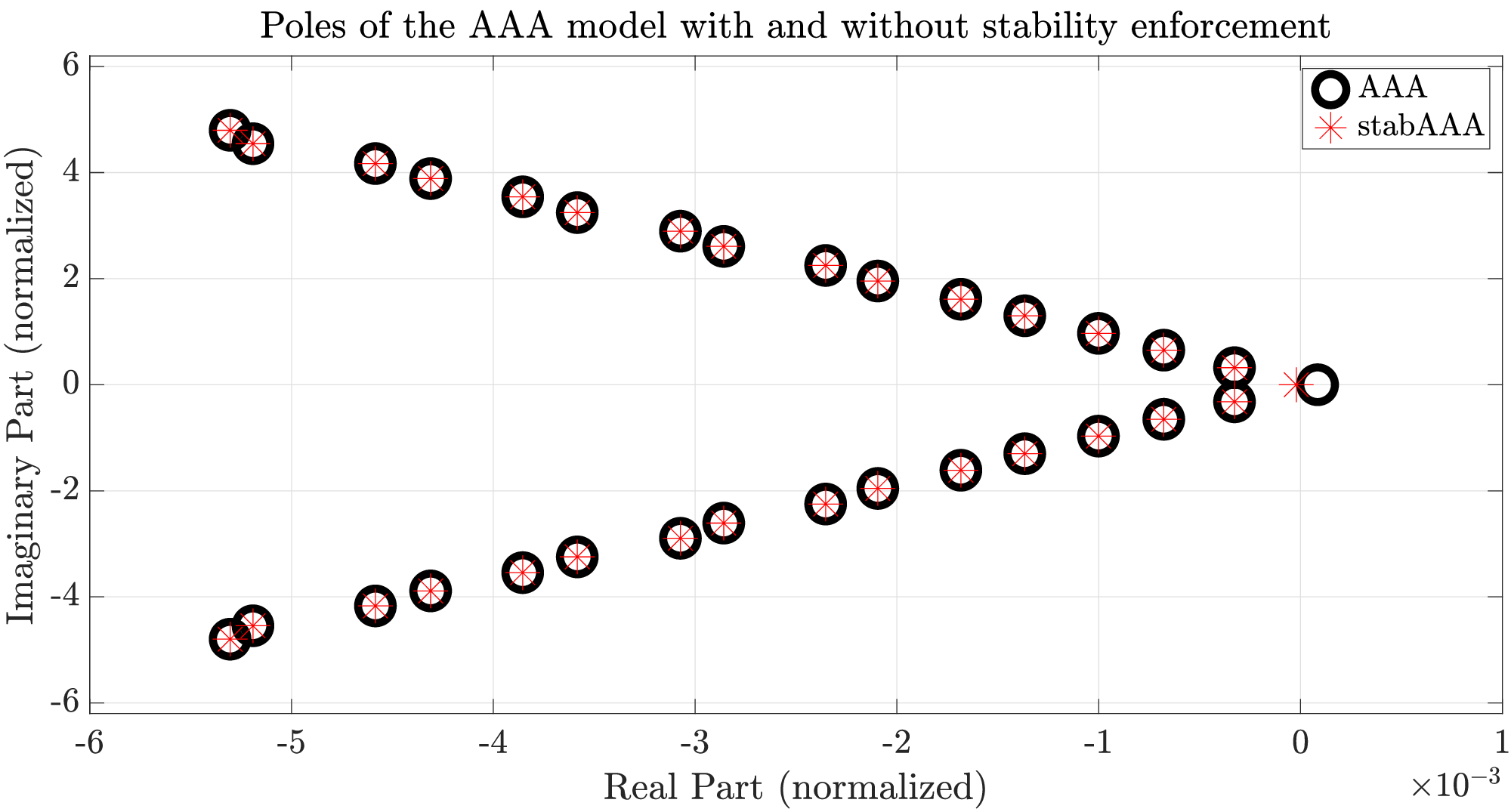}
    \includegraphics[width=0.8\columnwidth]{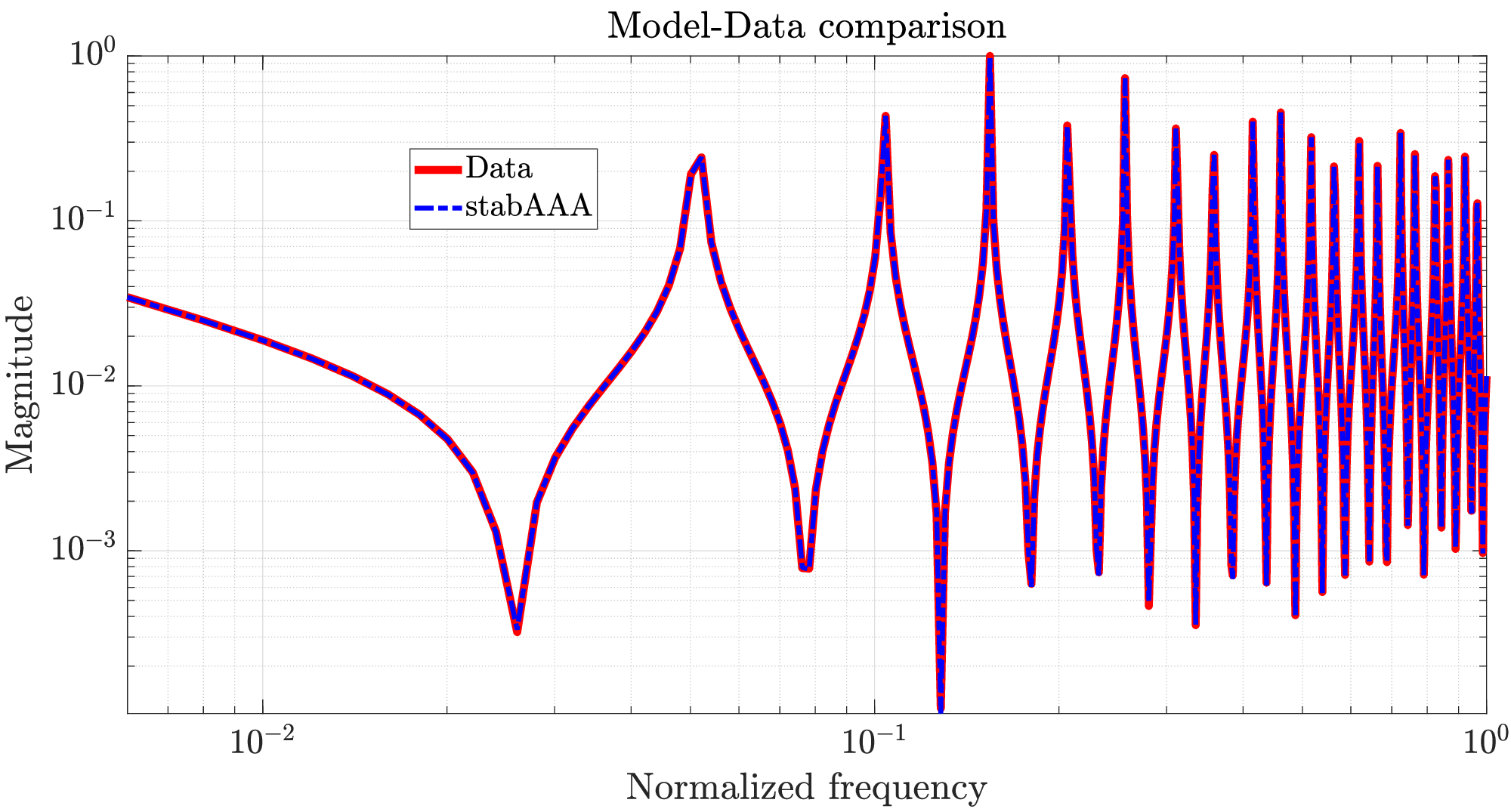}
    \includegraphics[width=0.8\columnwidth]{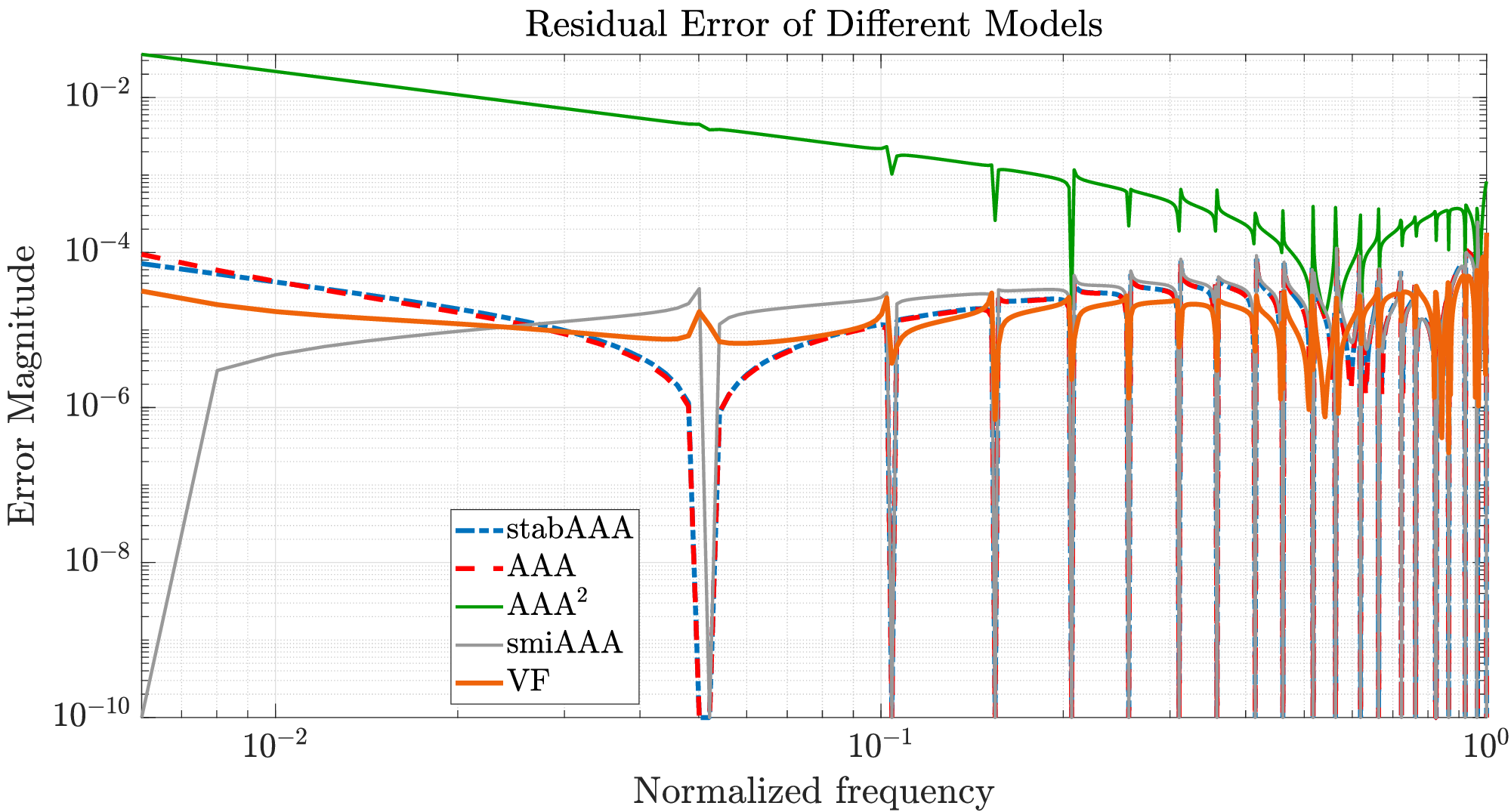}
    \caption{High Speed PCB link test case. Top panel: The plot reports a zoom on the dominant poles identified by the standard \texttt{AAA} algorithm (hollow black circles) and by \texttt{stabAAA} (red asterisks). Middle panel: Magnitude plot of the reference data against the response of the model obtained with \texttt{stabAAA}. Bottom Panel: Comparison between the residual error magnitude shown by ROMs obtained applying different approaches.}
    \label{fig:FittingSchuster}
\end{figure}

\subsection{International Space Station (ISS) 1R model}

In this section, we analyze a structural model for component 1R (Russian service module) of the ISS. This model was used for MOR purposes in, among others, the work in \cite{gugercin2001approximation}, and has become a standard MOR benchmark in the last two decades. The data for this test case are obtained from transfer function evaluations of the full order model along the imaginary axis $(j\lambda_v,h_v), v=1,\hdots 400,  \lambda_v \in [0.1,100]$~rad/s. 

Algorithm~\ref{al:aaa} was applied after data normalization with an admissible error threshold \mbox{$\epsilon=10^{-4}$}, that is met by the model $\hat  H_{ns}(s)$ after $k=31$ iterations with $E_\infty^{\hat H_{ns}}=5.62\times 10^{-5}$ and $E_2^{\hat H_{ns}}=1.94\times 10^{-4}$. The resulting model has one pair of unstable complex conjugate poles and one unstable real pole. The (normalized) dominant poles are depicted in Figure~\ref{fig:FittingISS} (top panel). Applying \texttt{stabAAA}, we obtain a model  $\hat  H_{s}(s)$ with error metrics $E^{\hat H_{s}}_\infty=  5.38\times 10^{-5}$ and $E^{\hat H_{s}}_2=1.96\times 10^{-4}$, meeting the desired accuracy requirements after $k=31$ iterations. For this test case, optimization problem~\eqref{eq:OptStep8} was solved in $6.5$ s on our hardware. The (normalized) dominant poles of this model are shown in the top panel of Fig.~\ref{fig:FittingISS}, while in the middle panel, we plot the response of  $\hat  H_{s}(s)$ against the reference frequency response. Also, in this case, Algorithm~\ref{al:aaa2} returns an accurate approximation of the underlying data. 

The bottom panel of Fig.~\ref{fig:FittingISS} and  Table~\ref{tab:overall} provide a visual and numerical comparison between the proposed approach and the competing methods. The performance of the proposed approach is equivalent to those of Algorithm~\ref{al:aaa} for all practical extents and better than those of other stable approaches derived from the \texttt{AAA} algorithm. When compared with \texttt{VF}, the method is slightly worse in the $E_2$ sense but better in the $E_\infty$ sense, as can be expected due to the different error minimization schemes driving the two algorithms.

\begin{figure}
    \centering   
    \includegraphics[width=0.8\columnwidth]{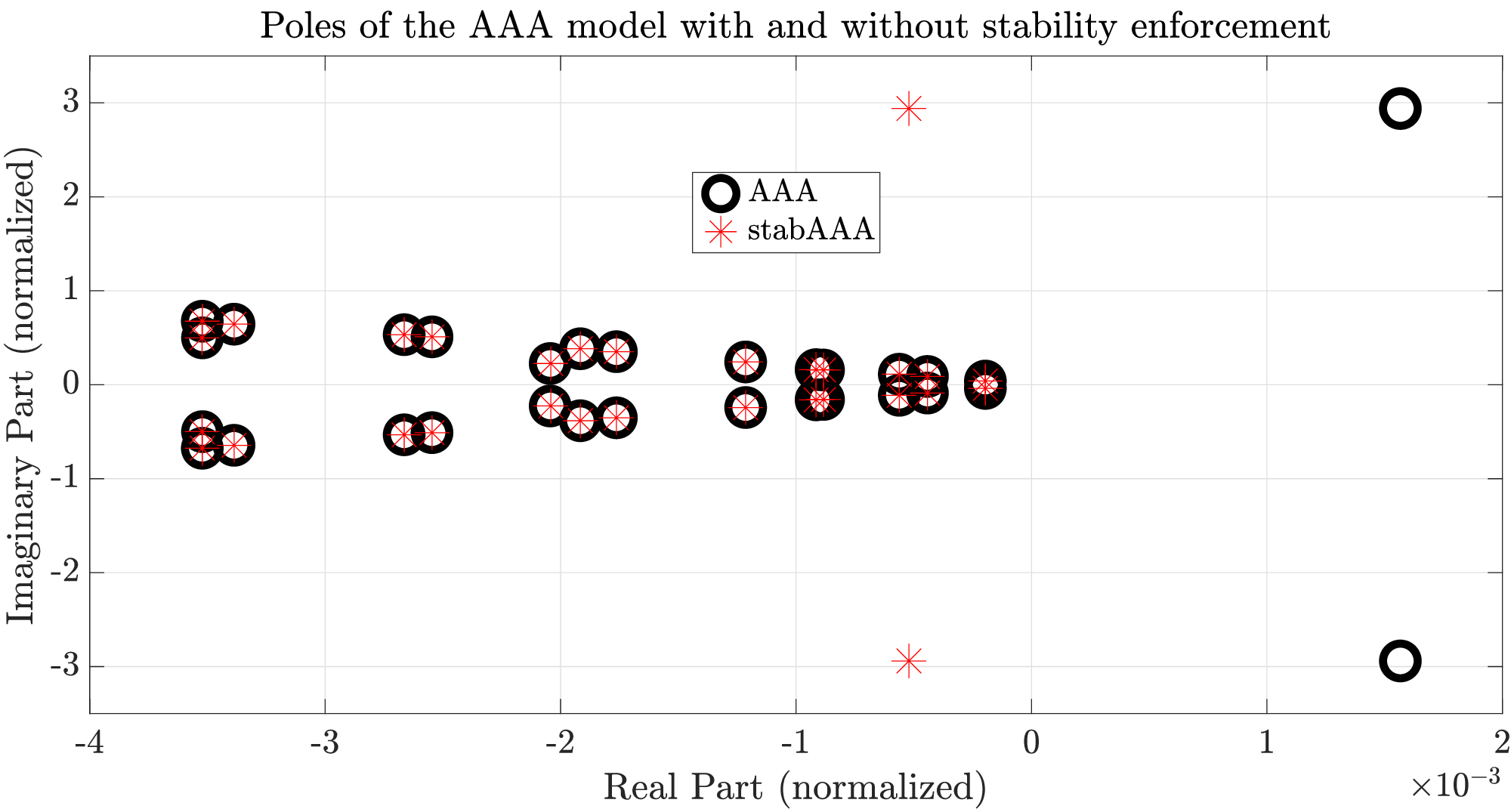}
    \includegraphics[width=0.8\columnwidth]{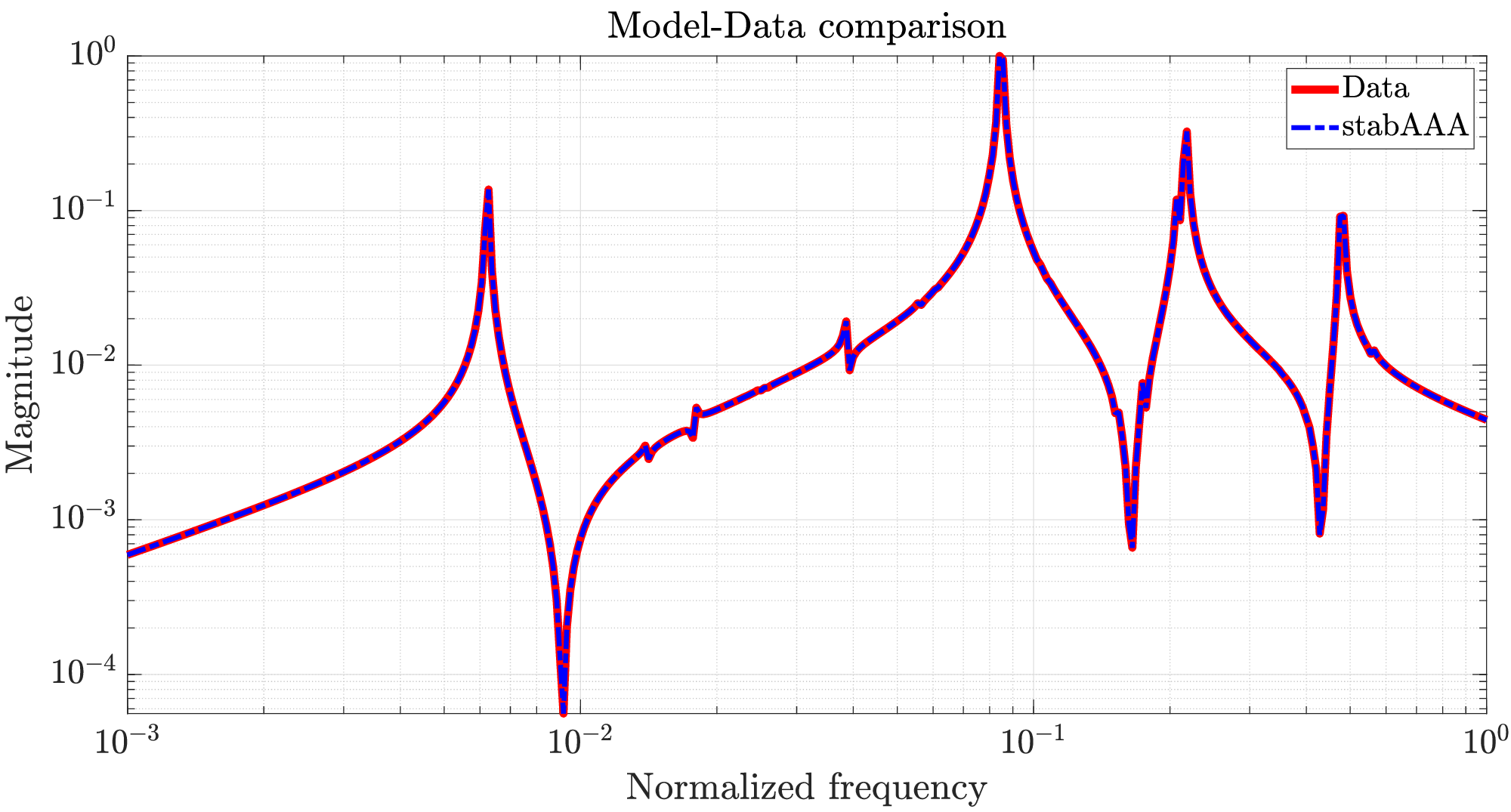}
    \includegraphics[width=0.8\columnwidth]{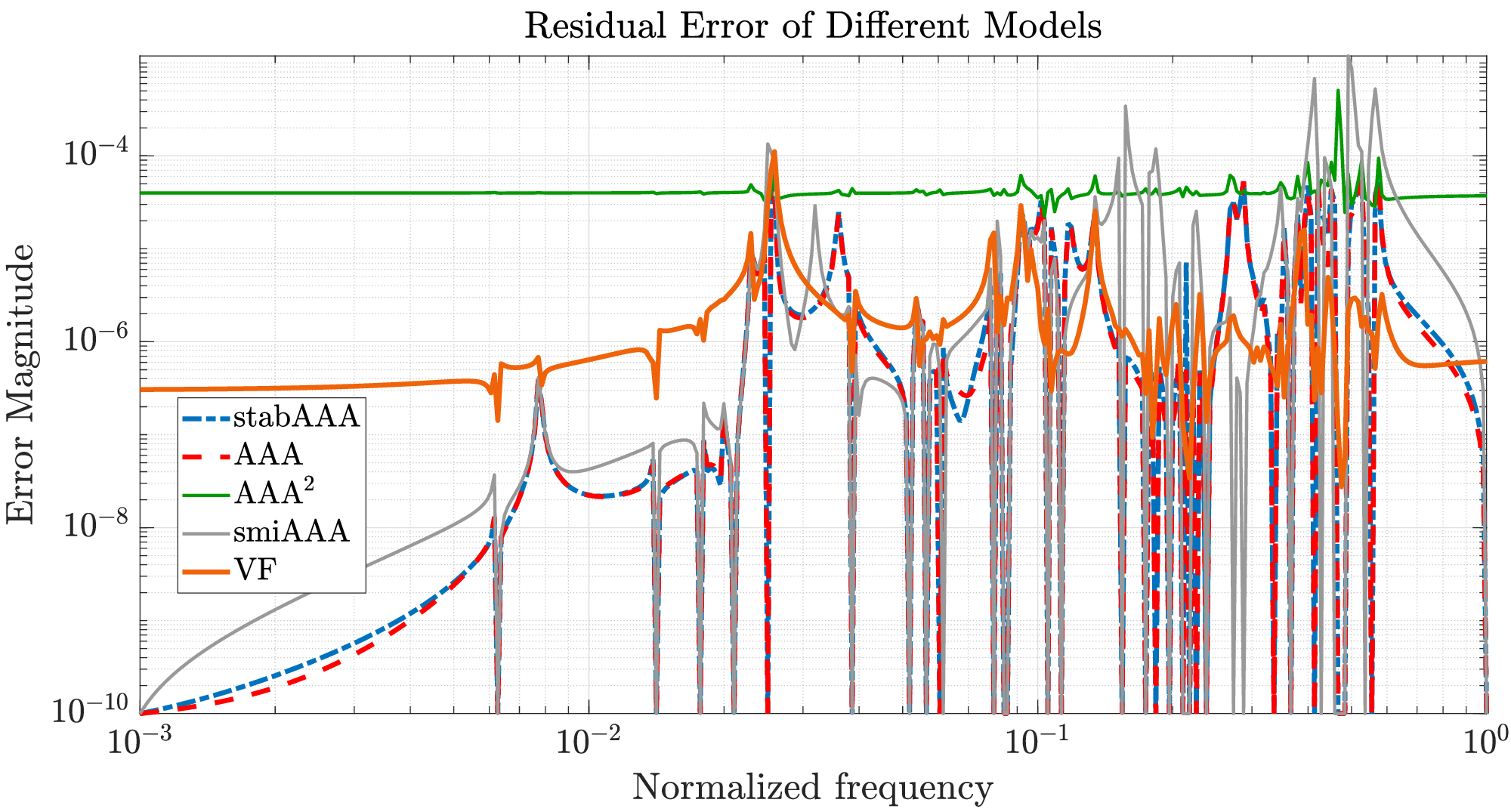}
    \caption{International Space Station test case. Top panel: zoom on the dominant poles identified by the standard \texttt{AAA} algorithm (hollow black circles) and by the stable \texttt{AAA} (red asterisks). Middle panel: magnitude plot of the reference data against the response of the \texttt{stabAAA} model. Bottom panel: comparison between the residual error magnitude obtained by applying different approaches.}
    \label{fig:FittingISS}
\end{figure}

\subsection{Acoustic Absorber}
The system considered in this section is an acoustic absorber, composed of a cavity with one side covered by a layer of poroelastic material. The structure details are available in~\cite{rumpler2014finite}, and the related data are taken from \cite{aumann2023structured}. The data samples for this example are $(j\lambda_v,h_v)=(j2\pi f_v,h_v), v=1,\hdots 901,  f_v \in [0.1,1]$~kHz. Normalizing the data, Algorithm~\ref{al:aaa} is used with $\epsilon=10^{-8}$. The algorithm stops after $k=16$ iterations, returning the model $\hat  H_{ns}(s)$ with $E_\infty^{\hat H_{ns}}=1.77\times 10^{-9}$ and $E_2^{\hat H_{ns}}=1.74\times 10^{-8}$, with $3$ unstable poles (one real and one complex conjugate pair), as shown in the top panel of Fig.~\ref{fig:FittingAbsorber}.

For this test case, we first run \texttt{stabAAA} setting $M_{max}=0$, in order to force the resulting model $\hat  H_{s}(s)$ to retain the same complexity of $\hat  H_{ns}(s)$. The resulting model shows the error metrics $E_\infty^{\hat  H_{s}}=1.92\times 10^{-6}$ and $E_2^{\hat  H_{s}}=2.55\times 10^{-5}$, that do not fulfill the required error tolerance. The optimization problem~\eqref{eq:OptStep8} is solved in $0.8$\,s. The magnitude plot of the data and of $\hat  H_{s}(s)$ are shown in the middle panel of Fig.~\ref{fig:FittingAbsorber}. We observe that for this test case, the proposed approach introduces significant accuracy degradation when compared to the standard (yet unstable) \texttt{AAA} algorithm when a model of the same complexity is required. In order to infer if this accuracy degradation is caused by the proposed convex relaxation of the exact non-convex optimization program~\eqref{eq:OptStep4}, we compare this model with the ones obtained using the competing approaches for even model complexity. 

\begin{figure}[h!]
    \centering
    \includegraphics[width=0.80\columnwidth]{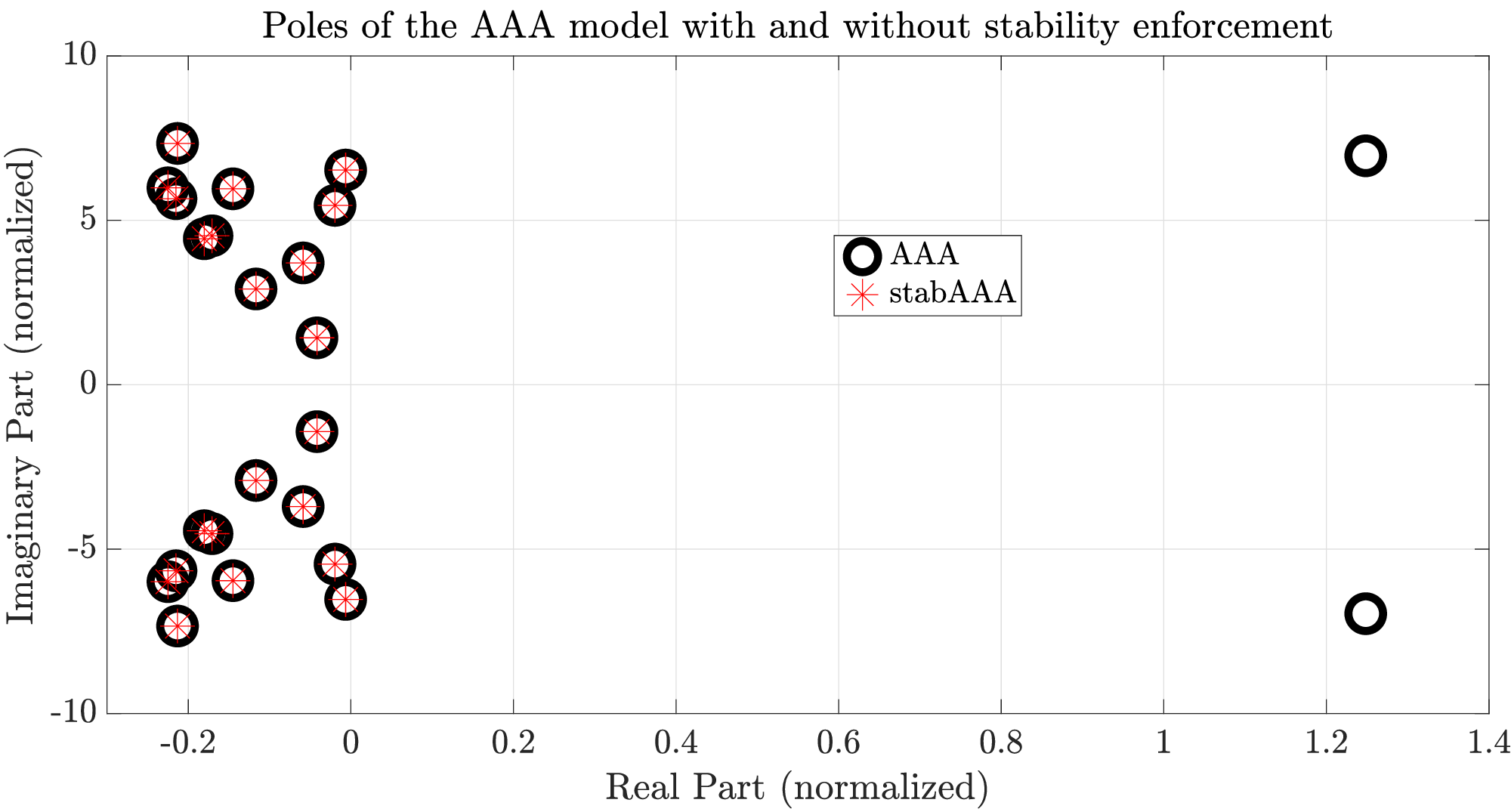}
    \includegraphics[width=0.8\columnwidth]{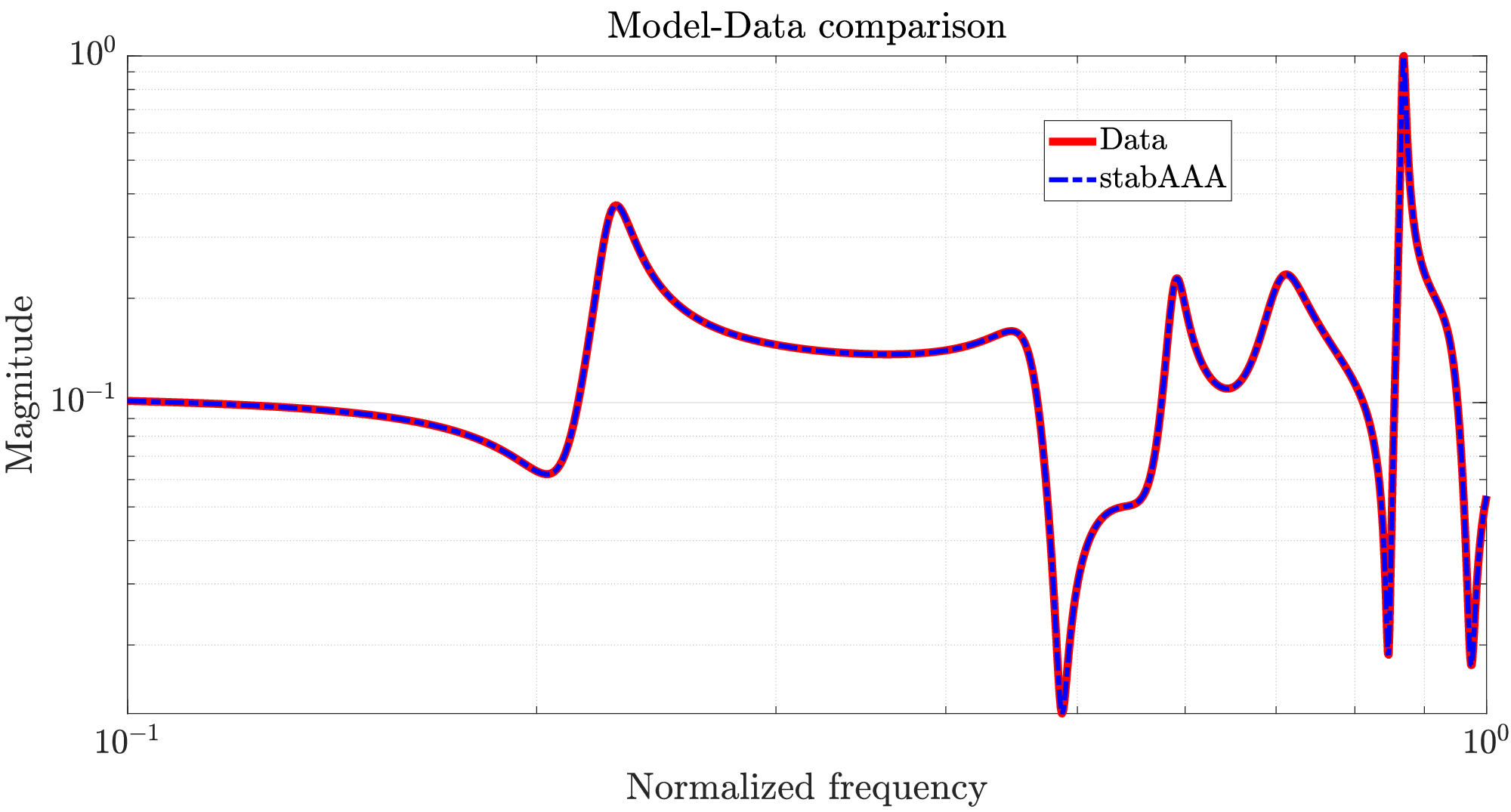}
    \includegraphics[width=0.8\columnwidth]{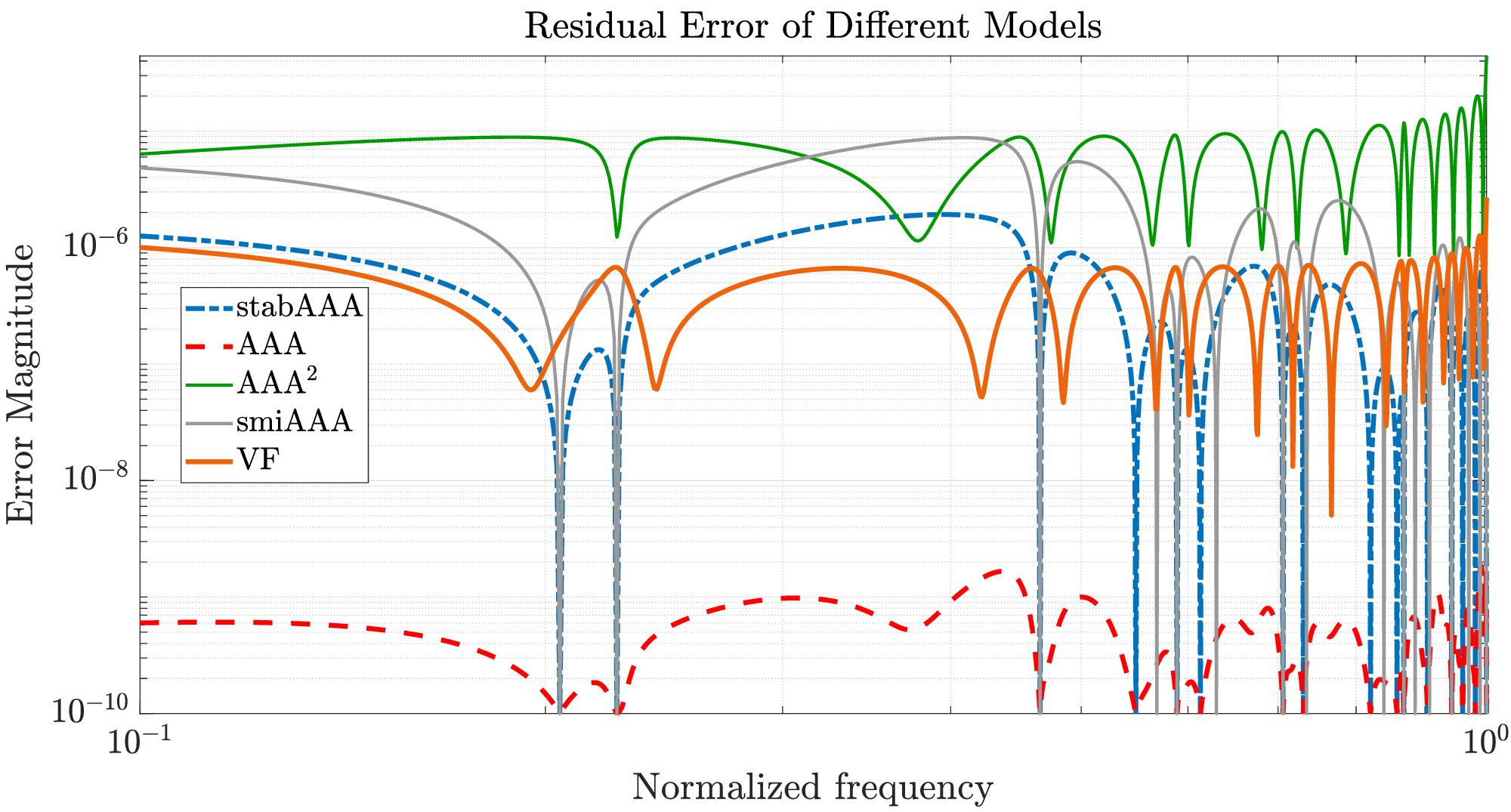}
    \caption{Acoustic absorber test case. Top panel: zoom on the dominant poles identified by the standard \texttt{AAA} algorithm (hollow black circles) and by the stable \texttt{AAA} (red asterisks). Middle panel: magnitude plot of the reference data against the response of the \texttt{stabAAA} model. Bottom panel: comparison between the residual error of the ROMs.}
    \label{fig:FittingAbsorber}
\end{figure}

The plot of the error magnitude for all of the generated models is given in the bottom panel of Fig.~\ref{fig:FittingAbsorber}, while a summary of the considered error metrics is given in Table~\ref{tab:overall}. The results confirm the conclusion drawn in the previous two examples, with the proposed approach performing practically the same as the \texttt{VF} while performing better than the other two stable approaches based on \texttt{AAA}, in particular in the sense of $E_2$. We conclude that the error degradation introduced by the proposed approach can be mostly related to the stability requirement itself rather than to our specific convex relaxation. Finally, in order to meet the prescribed accuracy requirements, we execute again Algorithm~\ref{al:aaa2}, with $M_{max}=5$ and $\theta=0.1$. With these settings, the algorithm stops after $k=17$ iterations, generating a model with error metrics $E_\infty= 3.62\times 10^{-11}$ and  $E_2= 4.1\times 10^{-10}$. In this case, the final result is the same as the one that would be obtained by running Algorithm~\ref{al:aaa} with $\epsilon=10^{-9}$, as with these settings the resulting model is stable on its own and the stability enforcement step is not needed.

\section{Conclusion}
\label{sec:conc}
In this paper, we proposed a novel formulation of the \texttt{AAA} algorithm which generates certified asymptotically stable ROMs of single-input single-output systems. The main enabling factor to achieve this goal is a set of novel algebraic conditions on the \texttt{AAA} barycentric weights that, whenever verified, guarantee the stability of the model by construction. Since these conditions are non-convex on the model coefficients, we propose to enforce them during model generation by solving a relaxed semi-definite programming problem that can be efficiently handled by any off-the-shelf convex optimization solvers. Experimental evidence shows that the performance of the proposed \texttt{stabAAA} algorithm is superior to those achievable by other state-of-the-art approaches based on \texttt{AAA}, and on a par with that of \texttt{VF}. When compared with this latter approach, \texttt{stabAAA} retains all the desirable features of the original \texttt{AAA} algorithm, including automated order selection, interpolation of the target transfer function at prescribed support points, and guaranteed convergence.

\section*{Acknowledgments}%
The first author would like to thank Professor Diego Regruto for the fruitful discussions about homogeneous least-squares problems.

\appendix

\section{Appendix}
\subsection{Explicit definition of $\IL^{(\ell)}_{\textrm{real}}$}\label{app:MatrixStructure}
By making use of the notation  $\Gamma^{(\ell)} = \{j\zeta_1,\ldots,j\zeta_{V-\ell}\}$, for all $1\leq i \leq \ell$ and $1 \leq l \leq V-\ell$, the elements of $\IL_{\textrm{real}}$ are defined as follows:
 $$
\left(\IL^{(\ell)}_{\textrm{real}}\right)_{l,2i-1}= \frac{-\texttt{Im} \left( H(j\zeta_l)\right)+\texttt{Im} \left( h_i^{(\ell)}\right)}{-\zeta_l+\lambda_i}+\frac{\texttt{Im} \left( H(j\zeta_l)\right)+\texttt{Im} \left( h_i^{(\ell)}\right)}{\zeta_l+\lambda_i}, \   
$$
$$
\left(\IL^{(\ell)}_{\textrm{real}}\right)_{l,2i}= \frac{-\texttt{Re} \left( H(j\zeta_l)\right)+\texttt{Re} \left( h_i^{(\ell)}\right)}{-\zeta_l+\lambda_i}+\frac{-\texttt{Re} \left( H(j\zeta_l)\right)+\texttt{Re} \left( h_i^{(\ell)}\right)}{\zeta_l+\lambda_i}, \ 
$$
$$
\left(\IL^{(\ell)}_{\textrm{real}}\right)_{l+V-\ell,2i-1}= \frac{\texttt{Re} \left( H(j\zeta_l)\right)-\texttt{Re} \left( h_i^{(\ell)}\right)}{-\zeta_l+\lambda_i}+\frac{-\texttt{Re} \left( H(j\zeta_l)\right)+\texttt{Re} \left( h_i^{(\ell)}\right)}{\zeta_l+\lambda_i}, \ \text{and} 
$$
$$
\left(\IL^{(\ell)}_{\textrm{real}}\right)_{l+V-\ell,2i}= \frac{-\texttt{Im} \left( H(j\zeta_l)\right)+\texttt{Im} \left( h_i^{(\ell)}\right)}{-\zeta_l+\lambda_i}+\frac{-\texttt{Im} \left( H(j\zeta_l)\right)+\texttt{Im} \left( h_i^{(\ell)}\right)}{\zeta_l+\lambda_i}.
$$
\subsection{Proof of \Cref{lemm:3.1}}
\label{sec:App1}

By using the structure of the matrices in \cref{ROM_SISO_improper}, and by denoting $\tilde{\bLambda}_s  = s \tilde{\bE}-\bLambda$, we get:
\begin{equation}
    \tilde{H}(s) = \tilde{\bC} (s \tilde{\bE} - \tilde{\bA})^{-1}  \tilde{\bB} = \tilde{\bC} (s \tilde{\bE} - \bLambda + \tilde{\bB} \tilde{\bR})^{-1}  \tilde{\bB} = \tilde{\bC} ( \tilde{\bLambda}_s + \tilde{\bB} \tilde{\bR})^{-1}  \tilde{\bB},
\end{equation}
and by making use of the Sherman-Morrison identity
, the function $\tilde{H}(s)$ can be rewritten as
\begin{equation}
	\label{eq:rom_sherm-morri}
	\tilde{H}(s) = \frac{\tilde{\bC} \tilde{\bLambda}_s^{-1} \tilde{\bB}}{1 + \tilde{\bR} \tilde{\bLambda}_s^{-1} \tilde{\bB}}.
\end{equation}
By using that $\tilde{\bLambda}_s  = s \hat{\bE}-\bLambda  = \text{diag}(\bLambda_s,-1)$, where $\bLambda_s = \text{diag}(s-j\lambda_1,s+j\lambda_1,\ldots,s-j\lambda_k,s+j\lambda_k)$ and by using notations $\tilde{\bB} = \begin{bmatrix}
        \hat{\bB}^T & 1
    \end{bmatrix}^T, \ \     \tilde{\bC} = \begin{bmatrix}
        \hat{\bC} & 0
    \end{bmatrix}, \ \ \tilde{\bR} = \begin{bmatrix}
        \bR & 1
    \end{bmatrix}$,
it follows that:
\begin{align}
    &1 + \tilde{\bR} \tilde{\bLambda}_s^{-1} \tilde{\bB} = 1 + \begin{bmatrix}
        \bR & 1
    \end{bmatrix} \left(\text{diag}(\bLambda_s,-1)\right)^{-1} \begin{bmatrix}
        \hat{\bB} \\ 1
    \end{bmatrix}  = 1 + \bR \bLambda_s^{-1} \hat{\bB} - 1 = \bR \bLambda_s^{-1} \hat{\bB}, \ \ \text{and} \\ &\tilde{\bC} \tilde{\bLambda}_s^{-1} \tilde{\bB} = \begin{bmatrix}
        \hat{\bC} & 0
    \end{bmatrix} \left(\text{diag}(\bLambda_s,-1)\right)^{-1} \begin{bmatrix}
        \hat{\bB} \\ 1
    \end{bmatrix} = \hat{\bC} \bLambda_s^{-1} \hat{\bB}.
\end{align}
By combining this result with that in (\ref{eq:rom_sherm-morri}), we obtain:
    $\tilde{H}(s) = \frac{\hat{\bC} \bLambda_s^{-1} \hat{\bB} }{\bR \bLambda_s^{-1} \hat{\bB}}$. We can write:
    
\begin{align*}
    &\hat{\bC} \bLambda_s^{-1} \hat{\bB} =  \begin{bmatrix} h_1& h_1^* & \cdots &  h_k & h_k^* \end{bmatrix} \begin{bmatrix}
        (s-j\lambda_1)^{-1} & 0 & 0 & \cdots & 0 \\
        0 & (s+j\lambda_1)^{-1} & 0 & \cdots & 0 \\
        \vdots & \vdots & \ddots & \vdots & \vdots \\
    \end{bmatrix} \begin{bmatrix} w_1 \\ w_1^* \\ \vdots \\ w_k \\ w_k^* \end{bmatrix} \\
    &= \sum_{i=1}^k \left( \frac{h_i w_i}{s -j \lambda_i} + \frac{(h_i w_i)^*}{s +j\lambda_i} \right), \\
      &\hat{\bR} \bLambda_s^{-1} \hat{\bB} =  \begin{bmatrix} 1& 1 & \cdots &  1 & 1 \end{bmatrix} \begin{bmatrix}
        (s-j\lambda_1)^{-1} & 0 & 0 & \cdots & 0 \\
        0 & (s+j\lambda_1)^{-1} & 0 & \cdots & 0 \\
        \vdots & \vdots & \ddots & \vdots & \vdots \\
    \end{bmatrix} \begin{bmatrix} w_1 \\ w_1^* \\ \vdots \\ w_k \\ w_k^* \end{bmatrix} \\
    &= \sum_{i=1}^k \left( \frac{ w_i}{s -j \lambda_i} + \frac{( w_i)^*}{s +j\lambda_i} \right).
\end{align*}
It follows that the transfer function $\tilde{H}(s)$ has the required barycentric form.

\subsection{Two useful state-space realizations}
\label{sec:App2}

Here, we provide some details on two commonly used state-space realizations, which are instrumental to the formulation of the proposed \texttt{stabAAA} framework.

\subsubsection{A complex-valued realization with unit input map}
\label{rem:remApp2}

Given the realization provided in \cref{ROM_SISO_improper}, we would like to re-write it so that all components in vector $\tilde{\bB}$ are ones, and the transfer function in \cref{eq:AAAReal} remains unchanged. Hence, write the new vectors as
    \begin{align}
    \begin{split}
    \tilde{\bR}_{\textrm{mod}} &= \begin{bmatrix}
     w_1 & w_1^* & \cdots & w_k & w_k^* & 1 
  \end{bmatrix} \in \IC^{1 \times (2k+1)}, \\ \tilde{\bC}_{\textrm{mod}} &=  \begin{bmatrix} h_1 w_1& h_1^* w_1^* & \cdots &  h_k w_k & h_k^* w_k^*& 0
		\end{bmatrix} \in \IC^{1 \times (2k+1)}, \\
    \tilde{\bB}_{\textrm{mod}} &= \begin{bmatrix}
     1 & 1 & \cdots & 1 & 1 & 1 
  \end{bmatrix}^T \in \IC^{(2k+1) \times 1}.
      \end{split}
    \end{align}
which replace corresponding quantities $\tilde{\bR}$, $\tilde{\bC}$, $\tilde{\bB}$ in~\cref{ROM_SISO_improper} by leaving the transfer function unchanged.

\subsubsection{Conversion to real-valued realizations}
\label{rem:remApp3}

By following the realness enforcement procedure in LF, i.e., as in \cite{ALI17}, we can also transform the matrices of the state-space representation of the \texttt{AAA} rational approximant as given in \cref{ROM_SISO_improper}, in order to be real-valued. To achieve this, let us define the following matrix:
\begin{align}
\bJ=\mbox{blkdiag}[
\overbrace{\bJ_2,\cdots,\bJ_2}^{k\,\mbox{terms}},1]\in\IC^{(2k+1)\times(2k+1)}, \quad \bJ_2=\frac{1}{\sqrt{2}}\left[\begin{array}{ cc }
  1 & \mathrm{j}  \\
  1 & -\mathrm{j}  \\
\end{array}\right]
\end{align}
where $\mbox{blkdiag}[\cdot]$ stacks its matrix arguments in block diagonal form.  

The transformed matrices can be explicitly written in all-real format, as follows: 
    \begin{align*}
    	\begin{split}
        \tilde{\bB}_{\text{real}} &= \bJ^*  \tilde{\bB} = \sqrt{2} \begin{bmatrix}
        \texttt{Re}(w_1) & \texttt{Im}(w_1) & \cdots & \texttt{Re}(w_k) & \texttt{Im}(w_k) & 1
    \end{bmatrix}^T \in \IR^{(2k+1) \times 1}, \\
     \tilde{\bC}_{\text{real}} &=  \tilde{\bC} \bJ = \sqrt{2} \begin{bmatrix}
        \texttt{Re}(h_1) & -\texttt{Im}(h_1) & \cdots & \texttt{Re}(h_k) & -\texttt{Im}(h_k) & 0
    \end{bmatrix}^T \in \IR^{1 \times (2k+1)}, \\
    \tilde{\bE}_{\text{real}} &= \text{diag}(1,1,\ldots,1,0) \in \IR^{(2k+1) \times (2k+1)},\
    \tilde{\bR}_{\text{real}} = \begin{bmatrix}
      \sqrt{2} & 0 & \cdots &  \sqrt{2} & 0 & 1
  \end{bmatrix} \in \IR^{1 \times (2k+1)}, \\
    \bLambda_{\text{real}} &= \bJ^*   \bLambda \bJ= \mbox{blkdiag}[
\overbrace{\bK_2 \lambda_1,\cdots,\bK_2 \lambda_k}^{k\,\mbox{terms}},1] \in \IR^{(2k+1) \times (2k+1)}, \ \ \bK_2 = \begin{bmatrix}
    0 & -1 \\ 1 & 0
\end{bmatrix},
    \end{split}
    \end{align*}
    and, finally, we also obtain that $ \tilde{\bA}_{\text{real}} = \bJ^* \bLambda \bJ - \bJ^* \tilde{\bB} \tilde{\bR} \bJ = \bLambda_{\text{real}} - \tilde{\bB}_{\text{real}} \tilde{\bR}_{\text{real}}   \in \IR^{(2k+1) \times (2k+1)}$.

\footnotesize

\printbibliography

\end{document}